\documentclass[10pt,reqno,twoside]{amsart}
\usepackage{mathrsfs}
\usepackage{amsfonts}
\usepackage{latexsym,amssymb,amsmath,amsthm}
\usepackage{amsmath,amsthm,amssymb}
\usepackage{amsmath}
\usepackage{a4wide}
\usepackage{booktabs} 
\usepackage{diagbox}  
\usepackage{multirow}
\usepackage[numbers,sort&compress]{natbib}
\usepackage{setspace}
\usepackage{fancyhdr}
\allowdisplaybreaks
\numberwithin{equation}{section}
\allowdisplaybreaks[4]

\numberwithin{equation}{section} \theoremstyle{plain}
\newtheorem{theorem}{Theorem}[section]%
\newtheorem{lemma}[theorem]{Lemma}%
\newtheorem{remark}[theorem]{Remark}%
\newcommand{\R}{{\mathbb R}}

\newcommand{\ds}{\displaystyle}

\newcommand{\vs}{\vspace}
\newcommand{\h}{\hspace}

\begin{document}
	\title
	[Existence of nontrivial solutions for critical biharmonic equations]
	{Existence of nontrivial solutions for critical biharmonic equations with logarithmic term }
	\maketitle
 \begin{spacing}{1.25}
	\begin{center}
		\author{Qihan He$^1$}\\
		 College  of  Mathematics and Information Sciences, Guangxi Center for Mathematical Research, Guangxi University, Nanning, 530003,  P. R. China
		\footnote{ Email: heqihan277@gxu.edu.cn.}
		
		\author{Juntao Lv$^2$ }\\
		School of Mathematics and Statistics, Wuhan University, Wuhan, 430072, P. R.  China,
		\footnote{Email: m18131375275@163.com.}
		
		\author{Zongyan Lv$^3$}\\
		School of Mathematical Sciences, Beijing Normal University, Beijing, 100875, P. R. China,
		\footnote{ Email: zongyanlv0535@163.com.}
		
		\author{Tong Wu$^{4^*}$}\\
		College  of  Mathematics and Information Sciences, Guangxi Center for Mathematical Research, Guangxi University, Nanning, 530003,  P. R. China,
		\footnote{ Email: tongwu97@126.com. }
		
	\end{center}
\end{spacing}


	
	\begin{abstract}
		In this paper, we consider the existence of nontrivial solutions to the following critical biharmonic problem with a logarithmic term
		\begin{equation*}
			\left\{ \ds\begin{array}{ll} \ds \Delta ^2 u=\mu \Delta u+\lambda u+|u|^{2^{**}-2}u+\tau u\log u^2, \ \ x\in\Omega,  \vs{.4cm}\\
				u|_{\partial \Omega }=\frac{\partial u}{\partial n}|_{\partial\Omega}=0,
			\end{array}
			\right.\h{1cm}
		\end{equation*}
		where $\mu,\lambda,\tau  \in \mathbb{R}$, $|\mu|+|\tau|\ne 0$, $\Delta ^2=\Delta \Delta $ denotes the iterated N-dimensional Laplacian, $\Omega \subset \mathbb{R}^{N}$ is a bounded domain with smooth boundary $\partial \Omega $, $2^{**}=\frac{2N}{N-4}(N\ge5)$ is the critical Sobolev exponent for the embedding $H_{0}^{2}(\Omega)\hookrightarrow L^{2^{**}}(\Omega)$ and  $H_0^2 (\Omega )$ is the closure of $C_0^ \infty (\Omega )$ under the norm $|| u ||:=(\ds \int_{\Omega}|\Delta u|^2)^\frac{1}{2}$. The uncertainty
of the sign of $s\log s^2$ in $(0,+\infty)$ has some interest in itself. To know which of the three terms $\mu \Delta u$, $\lambda u$ and $\tau u \log u^2$ has a greater influence on the existence of nontrivial weak solutions, we prove the existence of nontrivial weak solutions to the above problem for $N\ge5$ under some assumptions of $\lambda, \mu$ and $\tau$.
	\end{abstract}
	

	
	

	\section{ Introduction}

	We study the existence of nontrivial weak solutions for the following critical biharmonic problem with a logarithmic term:
	\begin{equation}\label{eqS1.1}
		\left\{ \ds\begin{array}{ll} \ds \Delta ^2 u=\mu \Delta u+\lambda u+|u|^{2^{**}-2}u+\tau u\log u^2, \ \ x\in\Omega,  \vs{.4cm}\\
			u|_{\partial \Omega }=\frac{\partial u}{\partial n}|_{\partial\Omega}=0,
		\end{array}
		\right.\h{1cm}
	\end{equation}
	\baselineskip=19pt
	under some assumptions of $\lambda, \mu$ and $\tau$, where $\Omega$ is a bounded smooth domain in $\R^N$, $\Delta ^2=\Delta \Delta $ is the square of the Laplacian which is called the biharmonic operator, $2^{**}=\frac{2N}{N-4}(N \ge 5)$ is the  critical exponent of the Sobolev embedding $H_{0}^{2}(\Omega)\hookrightarrow L^{2^{**}}(\Omega)$.

	As we all kown, $Pohoz\check{a}ev$ \cite{PSI} asserts that the following equation
	\begin{equation*}
		\left\{ \ds\begin{array}{ll} \ds -\Delta u=|u|^{2^*-2}u, \quad   &x \in\Omega,  \vs{.4cm}\\
			u=0,\quad &x\in \partial \Omega ,
		\end{array}
		\right.\h{1cm}
	\end{equation*}
	has no nontrivial solutions when $\Omega$ is star-shaped. But,
	as $Br\acute{e}zis$ and $Nirenberg$ have shown in \cite{BN}, a lower-order terms can reverse this circumstance. Indeed, they considered the following classical problem
	\begin{equation}\label{eqS1.2}
		\left\{ \ds
		\begin{array}{ll} \ds -\Delta u=|u|^{2^*-2}u+\lambda u, \quad   &x \in\Omega, \vs{.4cm}\\
			u=0,\quad &x\in \partial \Omega ,
		\end{array}
		\right.\h{1cm}
	\end{equation}
	\baselineskip=19pt
	and found out that the conditions for existence of positive solutions are different when $N=3$ and $N\ge4$.
	They showed that: (1) when $N\ge4$ and $\lambda\in(0,\lambda_1(\Omega))$, there  exists a positive solution for $\eqref{eqS1.2}$;
	(2) when $N=3$ and $\Omega$ is a ball, $\eqref{eqS1.2}$ has a positive solution if and only if $\lambda\in \left(\frac{1}{4}\lambda_1(\Omega),\lambda_1(\Omega)\right)$;
	(3) $\eqref{eqS1.2}$ has no solutions when $\lambda\le0$ and $\Omega$ is star-shaped,
	where $\lambda_1(\Omega)$ is the first eigenvalue of $-\Delta$  with zero Dirichlet boundary condition on $\Omega$.	After that, the second and higher-order critical problems have received wide attention and achieved abundant results(see \cite{BN,CSS,NSY,GDW,DY,DL,JC,LZY,LP,PS,YWS,ZH,ZY,HL} and the references therein).

	
	
	$Gu$, $Deng$ and $Wang$ \cite{GDW} studied the existence and non-existence of the nontrivial solutions for the following critical biharmonic equation
	\begin{equation}\label{eqS1.3}
		\left\{ \ds\begin{array}{ll} \ds \Delta^{2}u=\lambda u+|u|^{2^{**}-2}u, \quad   &x \in\Omega,  \vs{.4cm}\\
			u|_{\partial \Omega}=\frac{\partial u}{\partial n}|_{\partial \Omega }=0, \quad  &\lambda>0,
		\end{array}
		\right.\h{1cm}
	\end{equation}
	\baselineskip=19pt
	and showed the following results:
	(1) For $N\geq8$, $\lambda\in(0,\delta_1 (\Omega))$, the problem $\eqref{eqS1.3}$ has at least one nontrivial weak solution; (2) For $N=5, 6, 7$, $\Omega=B_R(0)\subset \R^N$, there exist two positive constants $\lambda^{*}(N)$ and $\lambda^{**}(N)$ with $\lambda^{**}(N)<\lambda^{*}(N)<\delta_1 (\Omega)$, such that the problem $\eqref{eqS1.3}$ has at least one nontrivial weak solution provided $\lambda\in(\lambda^{*}(N),\delta_1 (\Omega))$, and there is no nontrivial solutions to the problem $\eqref{eqS1.3}$ if $\lambda<\lambda^{**}(N)$, where $\delta_1 (\Omega)$ represents the first eigenvalue of $\Delta^2$ with the $Dirichlet$ boundary value condition on $\Omega$.
	
	
	Recently, $He$ and $Lv$ \cite{ HL} have discussed the existence of nontrivial solutions for the following critical biharmonic equation
	\begin{equation}\label{eqS1.4}
		\left\{ \ds\begin{array}{ll} \ds \Delta^{2}u=\mu \Delta u+\lambda u+|u|^{2^{**}-2}u, \quad   &x \in\Omega,  \vs{.4cm}\\
			u|_{\partial \Omega}=\frac{\partial u}{\partial n}|_{\partial \Omega }=0 ,
		\end{array}
		\right.\h{1cm}
	\end{equation}
	\baselineskip=19pt
	and showed that: The problem $\eqref{eqS1.4}$ has at least one nontrivial weak solution if one of the following conditions is true: (i) $N\geq5$, $\mu=0$ and $\lambda\in (\lambda^*(N),\delta_1(\Omega))$; (ii) $N\geq6$, $\mu \in (-\beta(\Omega), 0)$ and $\lambda <\frac{(\mu+\beta(\Omega))\delta_1(\Omega)}{\beta(\Omega)}$; (iii) $N=5$, $(\lambda,\mu)\in A:=\{ (\lambda,\mu)| \lambda\in(-\infty,
	\delta_1(\Omega)), \max\{ -\beta(\Omega), \frac{\beta(\Omega)} { \delta _1 (\Omega)}
	\lambda-\beta(\Omega)\}<\mu \} \bigcap B := \{ (\lambda,\mu)| \mu<0.0317\lambda-11.8681\}$. 
	
	The above results tell us that the lower term compared to the critical term can affect the existence of the solutions. It is obvious that in $\eqref{eqS1.1}$, $\tau u\log {u}^{2}$ is a lower order term at infinity compared to $|u|^{{2}^{\ast\ast }-2}u$. Therefore, motivated by the above mentioned results, we want to consider the influence of the logarithmic term $\tau u\log {u}^{2}$ on the existence of nontrivial weak solutions to the critical biharmonic equation, and study which of the three terms $\mu \Delta u$, $\lambda u$ and $\tau u \log u^2$ has a greater influence on the existence of nontrivial weak solutions to the problem $\eqref{eqS1.1}$.  The uncertainty
	of the sign of $s\log s^2$ in $(0,+\infty)$ makes this problem much more interesting.
	
	We define the variational functional corresponding to the equation $\eqref{eqS1.1}$ as
	\begin{equation*}
		I(u)=\ds \frac{1}{2}\ds \int{(|\Delta u|^2+\mu |\nabla u|^2-\lambda |u|^2)}dx-\frac{1}{2^{**}}\ds \int |u|^{2^{**}}dx-\frac{\tau}{2}\ds \int{u^2(\log {u^2}-1)}dx.
	\end{equation*}
	It is easy to see that the functional $I(u)$ is well-defined in  $H_0^2(\Omega)$.
	
	
	Before stating our results, we set
	$$
	\left\| u \right\|_1^2=\ds \int_\Omega |\Delta u|^2dx,
	$$
	$$
	\left\| u \right\|^{2}=\ds \int_\Omega \left(|\Delta u|^2+\mu |\nabla u|^2-(\lambda-\tau_+) |u|^2\right)dx,
	$$
	and
	$$\beta(\Omega):=\underset{\begin{smallmatrix}
			u\in H_{0}^{2}(\Omega )\setminus\{0\} \\
	\end{smallmatrix}}{\mathop{\inf }}\ \frac{\ds \int_\Omega |\Delta u|^2dx} {\int_{\Omega }{|\nabla u{{|}^{2}}dx}}.$$
	If $u\in H_{0}^{2}(\Omega )\setminus\{0\}$, then $u\in H_{0}^{1}(\Omega )\setminus\{0\}$ is apparent. According to the Poincar\'{e} inequality, we can see that
	$\int_{\Omega }|\nabla u|^2 dx >0$. So $\beta(\Omega)$ is well-defined.
	
	Denote
	$$
	\eta_{\lambda,\tau}:=\frac{\beta(\Omega)}{\delta_1(\Omega)}\max\{\lambda-\tau_+,0\}-\beta(\Omega),
	$$
	where $\delta_1 (\Omega)$ represents the first eigenvalue of $\Delta^2$ with  the $Dirichlet$ boundary value condition on $\Omega$.
	
	We will show that if $\mu>\eta_{\lambda,\tau}$, $\lambda<\delta_1(\Omega)+\tau
	_+$ and $\tau\in\mathbb{R}$, then the norm $\|u\|_1$ is equivalent to $\|u\|$ in $H_{0}^{2}(\Omega )$, in Lemma \ref{lem2.1} below. Hence, for the case of $\mu>\eta_{\lambda,\tau}$, $\lambda<\delta_1(\Omega)$ and $\tau<0$, by the $Sobolev$ embedding theorem and equivalent norm, there is a constant $c_{\lambda,\mu}>0$, depending on $\lambda$, $\mu$ and $\Omega$, such that  $|u|_{2^{**}}\leq c_{\lambda,\mu}\|u\|$. Then we let
	$$
	\tau^*(N):=-\frac{4}{N|\Omega|}c_{\lambda,\mu}^{-\frac{N-4}{4}}.
	$$

	
	
	Our main results can be stated as follows:
	
	\begin{theorem}\label{t1.1}
		There is at least a nontrivial weak solution to the problem $\eqref{eqS1.1}$ provided $\lambda<\delta_1(\Omega)+\tau_+$, and one of the following assumptions holds:
		
		$(i)_1$ $\eta_{\lambda,\tau}<\mu<0$,  $N\ge6$, $\tau>\tau^*(N)$;
		
		$(i)_2$ $\eta_{\lambda,\tau}<\mu<0$,  $N=5$ and $\tau^*(N)<\tau<0$;
		
		$(ii)_1$ $\mu=0$, $N\ge8$ and $\tau>0$;
		
		$(ii)_2$ $\mu=0$, $N=8$, $\tau^*(N)<\tau<0$ and $\frac{\sqrt{c_8}4e^{\frac{\lambda}{2\tau}+\frac{19}{6}}}{\rho_{max}^2}<1$, 		where $\rho_{max} :=sup\{r>0: \exists  x \in \Omega ~~s.t.~ B(x,r)\subset \Omega \}$;
		
		$(ii)_3$  $\mu=0$, $N=5,6,7$ and $\tau^*(N)<\tau<0$;
		
		$(iii)$ $\mu>0$, $N=5$ and  $\tau^*(N)<\tau<0$.
		
	\end{theorem}

	\begin{remark}
		In fact, by Lemmas \ref{lem2.2}, \ref{lem2.4}, \ref{lem2.7} and \eqref{eqS3.13}, we see that, when $N=6$, if $\mu>0$, $\tau^*(N)<\tau<0$, $\lambda<\delta_1(\Omega)$ and there exists a $\varphi\in C^{\infty}_0(\Omega,[0,1])$ with $\varphi(x)\equiv1$, $x\in B(0,\frac{\rho_{max}}{2})$, such that $2\mu+\tau\int_{0}^{\rho_{max}}\varphi^2rdr<0$, then the equation $\eqref{eqS1.1}$ has at least a nontrivial weak solution.
	\end{remark}
	\begin{remark}
		Theorem \ref{t1.1} tells us that when $N\ge6$, the $\mu\Delta u~(\eta_{\lambda,\tau}<\mu<0)$ term has a greater impact on the existence of nontrivial weak solutions to the equation $\eqref{eqS1.1}$ than $\lambda u$ and  $\tau u\log u^2$. But for $N=5$, $\tau u\log u^2~(\tau^*(N)<\tau<0)$ plays a leading role in the existence of nontrivial weak solutions to the equation $\eqref{eqS1.1}$. 
	\end{remark}
	
	\section{Some preliminary results}
	
	In this section, we devote to some preliminary Lemmas, which are crucial in our proof of the main results.
	
	
	
	
	\begin{lemma}\label{lem2.1}
		If $\mu>\eta_{\lambda,\tau}$, $\lambda<\delta_1(\Omega)+\tau
		_+$ and $\tau\in\mathbb{R}$, then the norm $\|u\|_1$ is equivalent to $\|u\|$ in $H_{0}^{2}(\Omega )$.
	\end{lemma}
	
	\begin{proof}
		A similar proof can be found in \cite{HL}. We omit it here.
	\end{proof}
	
	\begin{lemma}\label{lem2.2}
		Assume $N \ge 5$. If $\mu>\eta_{\lambda,\tau}$, $\lambda<\delta_1(\Omega)+\tau
		_+$ and $\tau>\tau^*(N)$,  then  the functional $I(u)$ has Mountain pass geometry structure:

		(i) there exist $\alpha$, $\rho$$>0$ such that $I(u)\ge\alpha$ for all $\|u\|=\rho$;
		
		(ii) there exists $\omega\in H_{0}^{2}(\Omega)$ such that $\|\omega\| \ge \rho$ and $I(\omega)<0$.
	\end{lemma}
	
	\begin{proof}
		A straightforward computation gives us
		$$
		\begin{array}{ll}
			I(u)=\ds \frac{1}{2} \int{\left(|\Delta u|^2+\mu |\nabla u|^2-(\lambda-\tau)|u|^2\right)}-\frac{1}{2^{**}}\int{|u|^{2^{**}}}-\frac{\tau}{2} \int{u^2 \log {u^2}}.
		\end{array}
		$$	
		
		If $\tau>0$, then
		$$
		I(u)=\ds\frac{1}{2}\|u\|^{2}-\frac{1}{{2^{**}}}|u|_{2^{**}}^{2^{**}}-\frac{\tau}{2} \int{u^2\log {u^2}}.
		$$
		It follows from the fact $s^2 \log s^2 \leq Cs^{2^{**}}$ for all $s\in  [1,+\infty)$ that
\begin{align*}
			\ds \tau \int{u^2\log {u^2}}&=\ds \tau \int_{u^2 \ge 1}{u^2 \log{u^2}}dx+\tau \ds \int_{u^2 \le 1}{u^2 \log{u^2}}dx\\
			&\le \tau \ds \int_{u^2 \ge 1}{u^2 \log{u^2}}dx\\
			&\le C \tau \ds \int_{u^2 \ge 1}{u^{2^{**}}}dx\\
			&\le C\tau |u|_{2^{**}}^{2^{**}}.
\end{align*}
		Hence
		$$
		\begin{array}{ll}
			I(u)&\ge \ds \frac{1}{2}\|u\|^2-\frac{1}{2^{**}}|u|_{2^{**}}^{2^{**}}-C\frac{\tau}{2}|u|_{2^{**}}^{2^{**}}\\[3mm]
			&\ge \ds \frac{1}{2}\|u\|^{2}-C\|u\|^{2^{**}}, \\[3mm]
		\end{array}
		$$
		which implies that there exist  $\alpha>0$ and $\rho>0$ such that $I(u)\ge \alpha>0$ for all $\|u\|=\rho$.
		
		If $\tau^*(N)<\tau<0$, then we have
		$$
		\begin{array}{ll}
			I(u)
			=\ds \frac{1}{2}\|u\|^2-\frac{1}{2^{**}}|u|_{2^{**}}^{2^{**}}-\frac{\tau}{2} \int{u^2 (\log {u^2}-1)}.\\[3mm]
		\end{array}
		$$
		
		We use the following fact that
		$$\frac{\tau}{2} \int{u^2 (\log {u^2}-1)}\le-\frac{\tau}{2}|\Omega|,$$
		which implies
		\begin{equation}\label{eqS2.1}
			I(u)\ge\frac{1}{2}\|u\|^{2}-\frac{c_{\lambda,\mu}}{{{2}^{**}}}\|u\|^{2^{**}}+\frac{\tau}{2}|\Omega|.
		\end{equation}
		
		Putting $\alpha:=\frac{2}{N}c_{\lambda,\mu}^{-\frac{N-4}{4}}+\frac{\tau}{2}|\Omega|$ and $\rho:=c_{\lambda,\mu}^{-\frac{N-4}{8}}$, then $\alpha>0$ and $\rho>0$. By $\eqref{eqS2.1}$,
		$$
		I(u)\ge 	\frac{1}{2}\rho^2-\frac{c_{\lambda,\mu}}{2^{**}}\rho^{2^{**}}+\frac{\tau}{2}|\Omega|=\frac{2}{N}c_{\lambda,\mu}^{-\frac{N-4}{4}}+\frac{\tau}{2}|\Omega|=\alpha>0,
		$$
		for all $\|u\|=\rho$.
		
		
		On the other hand, letting $u\in H^2_0(\Omega)\setminus\{0\}$ be a fixed function, then we have, as $ t\to+\infty$,
		$$
		\begin{array}{ll}
			I(tu)&=\ds \frac{t^2}{2}\ds \int{\left(|\Delta u|^2+\mu |\nabla u|^2-\lambda|u|^{2}\right)}-\frac{t^{2^{**}}}{2^{**}} \int{|u|^{2^{**}}}-\frac{\tau}{2}t^2 \int{u^2 (\log t^2 {u^2}-1)}\\
			&=\ds \frac{t^2}{2} \int{\left(|\Delta u|^2+\mu |\nabla u|^2-\lambda|u|^{2}\right)}-\frac{t^{2^{**}}}{2^{**}} \int|u|^{2^{**}}\\
			&-\ds\frac{\tau}{2}t^2\log t^2 \int u^2-\frac{\tau}{2}t^2 \int u^2(\log u^2-1) \to -\infty.
		\end{array}
		$$
		Therefore, it is not hard to find a function $\omega\in H^2_0(\Omega)$ such that $\|\omega\|\ge\rho$ and $I(\omega)<0$.
	\end{proof}

	\begin{remark}
		According to the  Mountain Pass Theorem without $PS$ compactness condition, we can deduce that there exists a sequence $\{u_n\}$ in $H_0^2(\Omega)$, such that as $n \rightarrow \infty$,
		\begin{equation}\label{eqS2.2}
			I(u_n)\rightarrow c_M~and~I'(u_n)\rightarrow 0,
		\end{equation}
		where
		$$
		c_M:=\underset{\gamma \in \Gamma }{\mathop{\inf }}\,\underset{t\in \left[ 0,1 \right]}{\mathop{\sup }}\,I(\gamma (t))>0,
		$$
		and
		$$
		\Gamma:=\Big\{\gamma\in C([0,1], H^2_0(\Omega)),~\gamma(0)= 0, ~I(\gamma(1))<0\Big\}.
		$$
	\end{remark}
	
	\begin{lemma}\label{lem2.4}
		Assume that $\mu>\eta_{\lambda,\tau}$, $\lambda<\delta_1(\Omega)+\tau
		_+$ and $\tau\in\mathbb{R}$. Then any $(PS)_d$ sequence $\{u_n\}$ must be bounded in $H^2_0(\Omega)$.
	\end{lemma}
	\begin{proof}
		According to the definition of $(PS)_d$ sequence, we have
		\begin{equation}\label{eqS2.3}
			\begin{array}{ll}
				I(u_n)&=\ds \frac{1}{2} \int\left(|\Delta u_n|^{2}+\mu |\nabla u_n|^{2}-\lambda |u_n|^{2}\right)-\frac{1}{2^{**}} \int |u_n|^{2^{**}}-\frac{\tau}{2} \int{u_n^2(\log {u_n^2}-1)}\\
				&=d+o_n(1),\\
			\end{array}
		\end{equation}
		and
		\begin{equation}\label{eqS2.4}
			\begin{array}{ll}
				\langle {I}'(u_n),u_n  \rangle&=\ds \int(|\Delta u_n|^2+\mu |\nabla u_n|^2-\lambda |u_n|^2)-\int|u_n|^{2^{**}}-\tau \int u_n^2 \log {u_n^2}\\
				&=o_n(1)\|u_n\|.
			\end{array}
		\end{equation}
		
		For $\tau=0$, it has been proved in \cite{HL}. 
		
		If $\tau>0$, by $\eqref{eqS2.3}$ and $\eqref{eqS2.4}$, we deduce
		$$
		\begin{array}{ll}
			d+{o_n}(1)+{{o}_{n}}(1)\|u_n\| &\ds \ge  I(u_n)-\frac{1}{2}\langle {I}'(u_n),{u_n} \rangle \\ [3mm]
			&=\ds(\frac{1}{2}-\frac{1}{2^{**}})\int|u_n|^{2^{**}}+\frac{ \tau }{2} \int|u_n|^2\\[3mm]
			&=\ds \frac{2}{N}|u_n|_{2^{**}}^{2^{**}}+\frac{\tau}{2}|u_n|_2^2\\[3mm]
			&\ge\ds \frac{\tau}{2}|u_n|^2_2,\\[3mm]
		\end{array}
		$$
		which implies that
		\[
		|u_n|_2^2 \le C+C\|u_n\|.
		\]
		
		Using the following inequality (see \cite{SW} or see the theorem 8.14 in \cite{LL})
		\begin{equation*}
			\ds \int u^2\log u^2\le \frac{a^2}{\pi}\ds \int|\nabla u|^2+(\log|u|^2_{2}-N(1+\log a))|u|^2_{2}, ~\forall a>0,~\forall u\in H_0^2(\Omega),
		\end{equation*}
		we know that for $n$ large enough,
		\begin{align*}
				&d+{o_n}(1)+{{o}_{n}}(1)\|{{u}_{n}}\|  \\
				\ge& I(u_n)-\frac{1}{2^{**}}\langle {I}'(u_n),{u_n} \rangle \\
				=& \frac{2}{N}\|u_n\|^2-\frac{2}{N}\tau \ds \int u_n^2 \log {u_n^2}+\frac{\tau}{2^{**}}|u_n|_2^2\\
				\ge& \frac{2}{N}\left\|u_{n}\right\|^2-\frac{2}{N}\tau\left[\frac{a^2}{\pi}\ds \int|\nabla u_n|^2+(\log\left|u_{n}\right|^2_{2}-N(1+\log a))\left|u_{n}\right|^2_{2}\right]+\frac{\tau}{2^{**}}|u_n|_2^2\\
				\ge& \frac{2}{N}\left\|u_{n}\right\|^2-\frac{2}{N}\tau \frac{a^2}{\pi \beta(\Omega)}\left\|u_{n}\right\|^2-\frac{2}{N}\tau |u_n|_2^2 \log{|u_n|_2^2}+\left[2(1+\log a)+\frac{1}{2^{**}}\right]\tau|u_n|_2^2\\
				\ge& \frac{2}{N}\left\|u_{n}\right\|^2-\frac{1}{N}\|u_n\|^2-C(|u_n|_2^{2-\delta}+|u_n|_2^{2+\delta})-C\|u_n\|\\
				\ge& \frac{1}{N}\|u_n\|^2-C(\|u_n\|^{\frac{2-\delta}{2}}+\|u_n\|^{\frac{2+\delta}{2}}+1)-C\|u_n\|,
		\end{align*}
		where we choose $a>0$ such that $\frac{a^2}{\pi \beta(\Omega)}\tau<\frac{1}{2}$ and $\delta\in (0,1)$.
		
		When $\tau<0$, then we have
			\begin{align*}
				d+{o_n}(1)+{{o}_{n}}(1)\|{{u}_{n}}\| &\ge I(u_n)-\frac{1}{2^{**}}\langle {I}'(u_n),{u_n} \rangle \\
				&= \frac{2}{N}\|u_n\|^2-\frac{2}{N}\tau \ds \int_\Omega u_n^2 \log {u_n^2}dx+\frac{\tau}{2}\int_\Omega {u_n}^2dx\\
				&=\frac{2}{N}\|u_n\|^2-\frac{2}{N}\tau \ds \int_\Omega u_n^2 \left(\log {u_n^2}-\frac{N}{4}\right)dx\\
				&=\frac{2}{N}\|u_n\|^2-\frac{2}{N}\tau \ds \int_\Omega u_n^2 \log \left( {e^{-\frac{N}{4}}u_n^2} \right)dx\\
				&\ge \frac{2}{N}\|u_n\|^2-\frac{2}{N}\tau \ds \int_{\{e^{-\frac{N}{4}}u_n^2\leq1\}} u_n^2 \log \left( {e^{-\frac{N}{4}}u_n^2} \right)dx\\
				&\ge \frac{2}{N}\|u_n\|^2-\frac{2}{N}\tau \ds \int_{\{e^{-\frac{N}{4}}u_n^2\leq1\}} -e^{\frac{N}{4}-1}dx\\
				&\ge \frac{2}{N}\|u_n\|^2+\frac{2\tau}{N} e^{\frac{N}{4}-1}|\Omega|.
			\end{align*}
		So $\{ {{u}_{n}} \}$ is bounded in $H_{0}^{2}(\Omega)$.
	\end{proof}

	\begin{lemma}\label{lem2.5}
		We assume that $\{u_n\}$ is bounded in $H_0^2(\Omega)$ such that $u_n$ converges to $u$ a.e. in $\Omega$, then
		\begin{equation}\label{eqS2.5}
			\lim\limits_{n\to\infty}\ds \int_\Omega  u^2_{n}\log u^2_{n}dx=\ds \int_\Omega u^2\log u^2dx.
		\end{equation}  	
	\end{lemma}
	
	\begin{proof}
		If $u\in H_{0}^{2}(\Omega )\setminus\{0\}$, then $u\in H_{0}^{1}(\Omega )\setminus\{0\}$ is apparent. According to \cite{DHPZ}, we can obtain the result.
		
	\end{proof}
	
	\begin{lemma}\label{lem2.6}
		Assume that $N\ge 5$, $\mu>\eta_{\lambda,\tau}$, $\lambda<\delta_1(\Omega)+\tau
		$ and $\tau\ge0$. If $d< \frac{2}{N} S^{\frac{N}{4}}$, then $I(u)$ satisfies the  $(PS)_d$ condition, where $S:~\triangleq inf \{|\Delta u|_2^2 :~u \in H^2 (\R^N),~|u|_{2^{**}} =1\}$ is the best Sobolev embedding constant for $H^2(\R^N)$ embedded into $L^{2^{**}}(\R^N)$.		
	\end{lemma}
	\begin{proof}
		We suppose that $\{u_{n}\}$ is a $(PS)_{d}$ sequence of the functional $I$. By Lemma \ref{lem2.4}, $\{u_{n}\}$ is bounded in  $H^2_{0}(\Omega)$. Therefore, there exists a subsequence of $\{u_n\}$ still denote by $\{u_n\}$, such that as $n \to +\infty$
		\begin{equation*}
			\left\{ \ds\begin{array}{ll} \ds u_n \rightharpoonup u \
				&\hbox{in} \ H^2_{0}(\Omega),  \vs{.1cm}\\
				u_n\rightarrow u \ &\hbox{in} \
				H_0^1(\Omega),  \vs{.1cm}\\
				u_n \rightarrow u \ &\hbox{a.e. in}\  \Omega.
			\end{array}
			\right.\h{1cm}
		\end{equation*}
		It follows from $\eqref{eqS2.2}$ that for each $\varphi\in C^{\infty}_{0}(\Omega)$,
		\[
		\langle I^{\prime}(u),\varphi \rangle=\lim\limits_{n\rightarrow \infty}\langle I^{\prime}(u_n),\varphi \rangle=0.
		\]
		Due to the arbitrariness of $\varphi$, we can see that $u$ is a weak solution to the equation
		\begin{gather*}
			\Delta^2 u=\mu \Delta u +\lambda u+\left|u\right|^{2^{**}-2}u+\tau u\log u^2,
		\end{gather*}
		which implies that
		\begin{equation}\label{eqS2.6}
			\ds \int(\left|\Delta u\right|^2+\mu |\nabla u|^2-\lambda u^2)-\ds \int |u|^{2^{**}}-\tau \ds \int u^2\log{u^2}=0,
		\end{equation}
		and
				\begin{align}
				I(u)\notag&=\ds\frac{1}{2} \ds \int(\left|\Delta u\right|^2+\mu |\nabla u|^2-\lambda u^2)-\frac{1}{2^{**}}\ds \int\left|u\right|^{2^{**}}-\frac{\tau}{2}\ds \int u^2(\log u^2-1)\\ \label{eqS2.7}
				&=\ds\frac{1}{2} \ds \int\left|u\right|^{2^{**}}+\frac{\tau}{2}\ds \int u^2\log u^2-\frac{1}{2^{**}}\ds \int\left|u\right|^{2^{**}}-\frac{\tau}{2}\ds \int u^2\log u^2+\frac{\tau}{2}\ds \int u^2\\ \notag
				&=\ds\frac{2}{N} \ds \int\left|u\right|^{2^{**}}+\frac{\tau}{2}\ds \int u^2\ge 0.
				\end{align}
		According to the definition of the $(PS)_{d}$ sequence, $Br\acute{e}zis-Lieb$ Lemma (see \cite{BL}) and (\ref{eqS2.5})-(\ref{eqS2.7}), we have
		\begin{equation}\label{eqS2.8}
			\begin{array}{ll}
				{o_n}(1)&=<I'({u_n}),u_n>\\[3mm]
				&=\ds \int( |\Delta u_n|^2+\mu |\nabla u_n|^2-\lambda |u_n|^2)-\ds \int|u_n|^{2^{**}}-\tau \ds \int {u_n}^2\log {u_n}^2\\[3mm]
				&=\ds \int\left|\Delta u\right|^2+\ds \int|\Delta (u_n-u)|^2+\mu \ds \int |\nabla u|^2-\lambda \ds \int u^2 \\[3mm]
				&-\ds \int\left|u\right|^{2^{**}}-\ds \int|u_n-u|^{2^{**}}-\tau\ds \int u^2\log u^2+{o_n}(1)\\[3mm]
				&=\ds \int|\Delta (u_n-u)|^2-\ds \int|u_n-u|^{2^{**}}+o_n(1),
			\end{array}
		\end{equation}
		and
		\begin{equation}\label{eqS2.9}
			\begin{array}{ll}
				d+{o_n}(1)&=I(u_n)\\[3mm]
				&=\ds\frac{1}{2}\ds \int( |\Delta u_n|^2+\mu |\nabla u_n|^2-\lambda |u_n|^2)-\frac{1}{2^{**}}\ds \int|u_n|^{2^{**}}\\[3mm]
				&-\ds\frac{\tau}{2} \ds \int {u_n}^2(\log {u_n}^2-1)\\[3mm]
				&=\ds\frac{1}{2}\ds \int\left|\Delta u\right|^2+\frac{1}{2}\ds \int (|\Delta (u_n-u)|^2+\mu  |\nabla u|^2-\lambda u^2)\\[3mm]
				&-\ds \frac{1}{2^{**}}\ds \int\left|u\right|^{2^{**}}-\ds\frac{1}{2^{**}}\ds \int|u_n-u|^{2^{**}}-\frac{\tau}{2}\ds \int u^2(\log u^2-1)+{o_n}(1)\\[3mm]
				&=\ds I(u)+\frac{1}{2}\ds \int|\Delta (u_n-u)|^2-\frac{1}{2^{**}} \ds \int|u_n-u|^{2^{**}}+o_n(1).
			\end{array}
		\end{equation}
		Setting $X_n=\ds \int |\Delta (u_n-u)|^2,~Y_n=\ds \int |(u_n-u)|^{2^{**}}$, it follows from $\eqref{eqS2.8}$ and $\eqref{eqS2.9}$ that
		\begin{equation}\label{eqS2.10}
			X_n-Y_n=o_n(1),
		\end{equation}
		and
		\begin{equation}\label{eqS2.11}
			\frac{1}{2}X_n-\frac{1}{2^{**}}Y_n=d+o_n(1)-I(u).
		\end{equation}
		Combining with $\eqref{eqS2.10}$ and $\eqref{eqS2.11}$, it is easy to see that $\{X_n\}$ and $\{Y_n\}$ are two bounded sequences. Thus there exist convergent subsequences still denoted by $\{X_n\}$ and $\{Y_n\}$. We may suppose that $X_n \rightarrow k$, as $n \rightarrow \infty$. Then it follows from $\eqref{eqS2.10}$ that $Y_n \rightarrow k$, as $n \rightarrow \infty$. According to the definition of $S$, we have
		\[
		|\Delta u|_2^2\ge S|u|^2_{2^{**}}, \forall u\in H^2(\R^N),
		\]
		and
		\[
		k+o_{n}(1)=X_n \ge S Y_n^{\frac{2}{2^{**}}}=Sk^\frac{N-4}{N}+o_{n}(1),
		\]
		which implies that if $k>0$, then $k\ge S^\frac{N}{4}$. So we have
		\[
		I(u)=d-(\frac{k}{2}-\frac{k}{2^{**}})=d-\frac{2}{N}k\le d-\frac{2}{N}S^{\frac{N}{4}}<0,
		\]
		which contradicts $\eqref{eqS2.7}$. Thus $k=0$, which implies that as $n \to \infty$
		\[
		X_n=\ds \int |\Delta(u_n-u)|^2\to 0.
		\]
		So $u_{n}$ strongly converges to $u$ in $H^2_{0}(\Omega)$.	
	\end{proof}
	
	\begin{lemma}\label{lem2.7}
		Assume that $N\ge 5$, $\mu>\eta_{\lambda,\tau}$, $\lambda<\delta_1(\Omega)
		$ and $\tau<0$. If $\{u_n\}$ is a $(PS)_d$ sequence of $I$ and $d\in (-\infty,0) \cup (0,\frac{2}{N} S^{\frac{N}{4}})$, then there exists a $u\in H^2_0(\Omega )\setminus\{0\}$ such that $u_n \rightharpoonup u$ weakly in $H^2_0(\Omega)$ and $u$ is a nontrivial weak solution of $\eqref{eqS1.1}$.
	\end{lemma}
	\begin{proof}
		Let $\{u_n\}$ be a $(PS)_d$ sequence of $I$,  By Lemma \ref{lem2.4}, we know that $\{u_n\}$ is
		bounded in $H^2_0(\Omega )$. Hence there exists $u\in H^2_0(\Omega )$ such that, up to a subsequence,
		\begin{equation*}
			\left\{ \ds\begin{array}{ll} \ds u_n \rightharpoonup u \
				&\hbox{in} \ H^2_{0}(\Omega),  \vs{.1cm}\\
				u_n\rightarrow u \ &\hbox{in} \
				H_0^1(\Omega),  \vs{.1cm}\\
				u_n \rightarrow u \ &\hbox{a.e. in}\  \Omega.
			\end{array}
			\right.\h{1cm}
		\end{equation*}
		Since $\langle I^{\prime}(u_n),\varphi \rangle \to 0$ as $n\to\infty$ for any $\varphi\in C^{\infty}_0(\Omega)$, $u$ is a weak solution to the equation
		\begin{equation*}
			\Delta^2 u=\mu \Delta u +\lambda u+\left|u\right|^{2^{**}-2}u+\tau u\log u^2.
		\end{equation*}
		Assume that $u=0$ and set $v_n=u_n-u$. Following the definition of $(PS)_d$ sequence and $Br\acute{e}zis-Lieb$ Lemma, we have
		$$\int|\Delta v_n|^2-\int|v_n|^{2^{**}}=o_n(1),$$
		and
		\begin{equation}\label{eqS2.12}
			\frac{1}{2}\int|\Delta v_n|^2-\frac{1}{2^{**}}\int|v_n|^{2^{**}}=d+o_n(1).
		\end{equation}
		If we let 
		$$\int|\Delta v_n|^2 \to k,~~as~n\to\infty,$$
		then
		$$\int|v_n|^{2^{**}} \to k,~~as~n\to\infty.$$
		It is easy to verify that $k>0$. Actually if $k=0$, then $\int|\Delta u_n|^2=\int|\Delta v_n|^2 \to 0$, which implies that $I(u_n)\to0$, contradicting to $d\ne0$. Going on as Lemma \ref{lem2.6}, we can see that $k\ge S^{\frac{N}{4}}$. So from $\eqref{eqS2.12}$, we obtain
		$$\frac{2}{N}S^{\frac{N}{4}}\le \frac{2}{N}k=(\frac{1}{2}-\frac{1}{2^{**}})k=d<\frac{2}{N}S^{\frac{N}{4}},$$
		which is a contradiction. So $u$ is a nontrivial weak solution of $\eqref{eqS1.1}$.
	\end{proof}

	\section{energy estimation}
	In this section, we estimate the energy level $d$, under some different assumptions on the parameters $\lambda,\mu,\tau$ and dimension $N$. According to Lemma \ref{lem2.6}, the variational functional $I$ satisfies the $(PS)_d$ condition provided $d< \frac{2}{N} S^{\frac{N}{4}}$. So, to get a mountain pass solution, we need to show
	$c_M< \frac{2}{N} S^{\frac{N}{4}}$, which, together with $c_M:=\underset{\gamma \in \Gamma }{\mathop{\inf }}\,\underset{t\in \left[ 0,1 \right]}{\mathop{\sup }}\,I(\gamma (t))\le \underset{t\ge 0}{\mathop{\sup }}\,I(tV_\varepsilon)$, implies that we need to prove that there is a function $V_\varepsilon\in H^2_0(\Omega)$ such that $\underset{t\ge 0}{\mathop{\sup }}\,I(tV_\varepsilon)<\frac{2}{N}S^\frac{N}{4}$ under some assumptions on $\lambda, \mu$ and $\tau$.
	
	For $\varepsilon>0$, $z\in\R^N$, we denote
	\begin{equation}\label{eqS3.1}
		{u_{\varepsilon,z} }(x)=\frac{{{[ N(N-4)({{N}^{2}}-4){{\varepsilon }^2} ]}^{\frac{(N-4)}{8}}}}{{{(\varepsilon +|x-z|^2)}^{\frac{(N-4)}{2}}}},
	\end{equation}
	and
	\begin{equation*}
		\begin{array}{ll}
			S=\ds \inf \{ \frac{|\Delta u|_{2}^{2}}{|u|^2_{2^{**}}}:u\in {H^2}({\R^N})\setminus\{0\} \}.
		\end{array}
	\end{equation*}
	Following the idea in \cite{EFJ,VL}, the set $\{u_{\varepsilon,z}:\varepsilon>0,z\in\R^N\}$ contains all positive solutions of
	$${\Delta ^2}u=u^{2^{**}-1} \quad in\quad \R^N.$$
	It is easy to verify that
	$$
	|\Delta {u_{\varepsilon,z} }|_2^2=|{u_{\varepsilon,z} }|^{2^{**}}_{2^{**}}=S^{\frac{N}{4}}.
	$$
	
	Without loss of generality, we may assume that $0\in \Omega$ and $0$ is the geometric center of $\Omega$, that is $\rho_{max}=dist(0,\partial \Omega)$. We take a cut-off function $\varphi \in C_{0}^{\infty}( \Omega,[ 0,1 ] )$ satisfying that $\varphi (x)=1$ for $0\le|x|\le \rho$, $\varphi (x)\in (0,1)$ for $\rho <|x|<2\rho $ and $\varphi (x)=0$ for $|x|\ge 2\rho $. Then we set
	$$
	{u_{\varepsilon} }(x) \triangleq {u_{\varepsilon,0} }(x)=\frac{{{[ N(N-4)({{N}^{2}}-4){{\varepsilon }^2} ]}^{\frac{(N-4)}{8}}}}{{{(\varepsilon +|x|^2)}^{\frac{(N-4)}{2}}}},
	$$
	and
	\begin{equation*}
		V _\varepsilon (x)=\varphi (x)  u_{\varepsilon} (x).
	\end{equation*}
	\begin{lemma}\label{lem3.1}
		We have that $V_\varepsilon$ satisfies the following estimates, as $\varepsilon\to0^+$:
		\begin{equation*}
			\begin{array}{ll}
				\ds \int_\Omega|\Delta V _\varepsilon|^2={S^\frac{N}{4}}+O(\varepsilon ^\frac{N-4}{2}),\quad\quad\quad for~N\ge5,
			\end{array}
		\end{equation*}
		\begin{equation*}
			\begin{array}{ll}
				\ds \int_\Omega|V _\varepsilon|^{2^{**}}=\left\{ \begin{matrix}
					\ds S^\frac{N}{4}+O({\varepsilon }^\frac{N}{2}),&for~N\ge8,\\[3mm]
					\ds S^\frac{N}{4}+O({\varepsilon }^\frac{N-4+\delta}{2}),&for~5\leq N\leq7,\\[3mm]
				\end{matrix} \right.
			\end{array}
		\end{equation*}
		and
		\begin{equation}\label{eqS3.2}
			\begin{array}{ll}
				\ds \int_\Omega|V _\varepsilon|^{2}=\left\{ \begin{matrix}
					\ds c_NK_2 \varepsilon^2+O(\varepsilon^\frac{N-4}{2}),&for~N>8, \\[3mm]
					\ds \frac{1}{2}c_8\omega_8\varepsilon^2 \log \frac{1}{\varepsilon}+O(\varepsilon^2),&for~N=8, \\[3mm]
					\ds O(\varepsilon^{\frac{N-4}{2}}),&for~5\leq N\leq7,\\[3mm]
				\end{matrix} \right.
			\end{array}
		\end{equation}
		where $c_N=(N(N-4)(N^2-4))^\frac{N-4}{4}$, $K_2=\ds \int_{\R^N}\frac{1}{(1+|y|^2)^{N-4}}dy$, $0<\delta<1$ and $\omega_N$ denotes the area of the unit sphere surface in $\R^N$.
	\end{lemma}
	\begin{proof}
		The proof can be found in \cite{GDW} and \cite{HL}. So we omit it here.
	\end{proof}
	\begin{lemma}\label{lem3.2}
		We have, as $\varepsilon\to0^+$
		\begin{equation*}
			\begin{array}{ll}
				\ds \int_\Omega|\nabla V_\varepsilon |^2=\left\{ \begin{matrix}
					C_N{K_1}\varepsilon +O({\varepsilon }^{\frac{N-4}{2}}),&for~N\ge8,\\[3mm]
					\ds 9c_7\omega_7\varepsilon^{\frac{3}{2}}\int_{0}^{2\rho}\varphi^2\frac{r^8}{(\varepsilon+r^2)^5}dr+O(\varepsilon^{\frac{3}{2}})\ge \ds \frac{9}{32}c_7\omega_7\varepsilon+O(\varepsilon^{\frac{3}{2}}),&for~N=7, \\[3mm]
					\ds 2c_6\omega_6\varepsilon\log\frac{1}{\varepsilon}+O(\varepsilon),&for~N=6,\\[3mm]
					\ds O(\varepsilon^{\frac{1}{2}}),&for~N=5,\\[3mm]
				\end{matrix} \right.
			\end{array}
		\end{equation*}
		where $C_N=c_N(N-4)^2$ and $K_1= \ds \int_{\R^N}\frac{|y|^2}{(1+|y|^2)^{N-2}}dy$. 
	\end{lemma}
	\begin{proof}
		The case of $N\ge8$ has been proved in \cite{HL}. Thus we only prove the case of $5\leq N\leq7$. Followin the definition of $V_\varepsilon$, we have
		\begin{equation*}
			\begin{array}{ll}
				\ds\int_ \Omega  |\nabla V_\varepsilon |^2 dx& =\ds\int_\Omega |u_\varepsilon \nabla \varphi +\varphi \nabla u_\varepsilon |^2 dx\\[3mm]
				&=\ds\int_\Omega \big( (\nabla \varphi )^2  {u_\varepsilon }^2 +2\varphi  u_\varepsilon  \nabla\varphi  \nabla u_\varepsilon  +(\nabla u_\varepsilon )^2 {\varphi ^2} \big)dx \\[3mm]
				
				
				&=\ds c_N  \omega _N  \varepsilon^\frac{N-4}{2} (N-4)^2 \ds\int_0^{\rho}\varphi ^2 \frac {r^{N+1}} {(\varepsilon + r^2)^{N-2}}dr +O(\varepsilon^{\frac{N-4}{2}}),
			\end{array}
		\end{equation*}
		where we have used the facts that 
		$$\int_{\Omega\setminus\ B_\rho(0)}(\nabla \varphi )^2  {u_\varepsilon }^2dx\le C\varepsilon^{\frac{N-4}{2}},$$
		and
		$$\int_{\Omega\setminus\ B_\rho(0)}2\varphi  u_\varepsilon  \nabla\varphi  \nabla u_\varepsilon dx\le C\varepsilon^{\frac{N-4}{2}}.$$
		
		For $N=7$, we get that, as $\varepsilon\to0^+$
		\begin{align*}
				\ds\int_ \Omega  |\nabla V_\varepsilon |^2 dx
				&=\ds 9c_7\omega_7\varepsilon^{\frac{3}{2}}\int_{0}^{2\rho}\varphi^2\frac{r^8}{(\varepsilon+r^2)^5}dr+O(\varepsilon^{\frac{3}{2}})\\
				&=\ds 9c_7\omega_7\varepsilon^{\frac{3}{2}}\int_{0}^{\rho}\frac{r^8}{(\varepsilon+r^2)^5}dr+O(\varepsilon^{\frac{3}{2}})\\				
				&\ge \ds 9c_7\omega_7\varepsilon^{\frac{3}{2}}\cdot\frac{1}{\sqrt{\varepsilon}}\int_{1}^{\frac{\rho}{\sqrt{\varepsilon}}}\frac{r^8}{(1+r^2)^5}dr+O(\varepsilon^{\frac{3}{2}})\\
				&\ge \ds -\frac{9}{32}c_7\omega_7\varepsilon\cdot\frac{1}{r}{\Big\arrowvert}^{\frac{\rho}{\sqrt{\varepsilon}}}_1+O(\varepsilon^{\frac{3}{2}})\\
				&=\ds \frac{9}{32}c_7\omega_7\varepsilon+O(\varepsilon^{\frac{3}{2}}).
			\end{align*}
		
		For $N=6$, we obtain that, as $\varepsilon\to0^+$
		\begin{equation*}
			\begin{array}{ll}
				\ds\int_ \Omega  |\nabla V_\varepsilon |^2 dx
				&=\ds 4c_6\omega_6\varepsilon \int_{0}^{\rho}\varphi^2\frac{r^7}{(\varepsilon+r^2)^4}dr+O(\varepsilon)\\[3mm]
				&=\ds 4c_6\omega_6\varepsilon\left(\int_0^{\rho} \frac{r^7}{(\varepsilon +r^2)^4}dr+\int_0^{\rho}\frac{\varphi^2 -1}{r}(1-\frac{\varepsilon^4 +4\varepsilon^3 r^2+6\varepsilon^2 r^4+4\varepsilon r^6 }{(\varepsilon +r^2)^4})dr\right)+O(\varepsilon)\\[3mm]
				
				&=\ds 4c_6\omega_6\varepsilon \int_{0}^{\rho}\frac{r^7}{(\varepsilon+r^2)^4}dr+O(\varepsilon)\\[3mm]
				&\ds=2c_6\omega_6\varepsilon\int_0^{\rho^2}{r^3}d( -\frac{1}{3} (\varepsilon +r)^{-3} )+O(\varepsilon)\\
				&\ds=2c_6\omega_6\varepsilon\int_0^{\rho^2} r^2d( -\frac{1}{2} (\varepsilon +r)^{-2} ) +O(\varepsilon)\\[2mm]
				
				&\ds=2c_6\omega_6\varepsilon\int_0^{\rho^2} r d ( -(1+\varepsilon )^{-1} )+O(\varepsilon) \\[2mm]
				&=\ds2c_6\omega_6\varepsilon \log \frac{1}{\varepsilon} +O(\varepsilon).
			\end{array}
		\end{equation*}
		
		For $N=5$, we have that, as $\varepsilon\to0^+$
		\begin{equation*}
			\begin{array}{ll}
				\ds\int_ \Omega  |\nabla V_\varepsilon |^2 dx
				&=\ds c_5\omega_5\varepsilon^{\frac{1}{2}} \int_{0}^{\rho}\varphi^2\frac{r^6}{(\varepsilon+r^2)^3}dr+O(\varepsilon^{\frac{1}{2}})\\[3mm]
				&=\ds c_5\omega_5\varepsilon^{\frac{1}{2}}\int_0^{\rho}  \varphi^2 \left( 1-\frac{1}{(\varepsilon +r^2)^3} (3r^4 \varepsilon +3r^2\varepsilon^2 +\varepsilon^3) \right)dr+O(\varepsilon^{\frac{1}{2}})\\[3mm]
				&=\ds O(\varepsilon^{\frac{1}{2}}).
			\end{array}
		\end{equation*}
		%
	\end{proof}

	\textbf{The case of} $\mathbf{N> 8:}$
	
	\begin{lemma}\label{lem3.3}
		If $N>8$, then we have that, as $\varepsilon\to 0^+$:
		$$
		\ds \int_\Omega V_\varepsilon^2 \log V_\varepsilon ^2=\ds c_N \frac{N-4}{2} K_2 \varepsilon^2 \log \frac{1}{\varepsilon}+O(\varepsilon^2),
		$$
		where $c_N$ and $K_2$ have been given in Lemma \ref{lem3.1}.
	\end{lemma}
	\begin{proof}
		For $N>8$, we have
		$$
		\begin{array}{ll}
			\ds \int_\Omega V_\varepsilon^2 \log V_\varepsilon ^2&=\ds \int_\Omega \varphi^2 u_\varepsilon^2 \log \varphi^2 dx+ \ds \int_\Omega \varphi^2 u_\varepsilon^2 \log u_\varepsilon ^2 dx \\
			&=\ds \int_\Omega \varphi^2 u_\varepsilon^2 \log \varphi^2 dx+\ds \int_{B_\rho (0)}\varphi^2 u_\varepsilon^2 \log u_\varepsilon ^2dx+\ds \int_{\Omega \setminus\ B_\rho (0)}\varphi^2 u_\varepsilon^2 \log u_\varepsilon ^2dx\\
			&\triangleq I+II+III.
		\end{array}
		$$	
		Since $\left|s^2\log s^2\right| \le C$ for $0< s \le 1$, we have that
		$$
		|I|\le C \ds \int u_\varepsilon^2=O(\varepsilon^2).
		$$
		Due to $N>8$ we can find a $\ds \delta\in(0,\frac{1}{5})$ such that $\ds 	\frac{1}{2}(N-4)(1-\delta)\ge2$. So it follows from $\left|s\log s\right| \le C_1s^{1-\delta}+C_2 s^{1+\delta}$ for all $s>0$ that
\begin{align*}
			|III|&\le \ds \int_{\Omega \setminus{B_\rho (0)}}\left|u_\varepsilon^2 \log u_\varepsilon ^2\right|dx\\
			&\le C\ds \int_{\Omega \setminus\{B_\rho (0)\}} \left(u_\varepsilon^{2(1-\delta)}+u_\varepsilon^{2(1+\delta)}\right)dx\\
			&\le \ds C|\Omega|\left(\varepsilon^{\frac{N-4}{2}(1-\delta)}+\varepsilon^{\frac{N-4}{2}(1+\delta)}\right)\\
			&=\ds O(\varepsilon^2),
\end{align*}
		and
		$$
		\begin{array}{ll}
			II&=\ds \int_{B(0,\rho)}u_\varepsilon^2 \log u_\varepsilon ^2 dx\\[4mm]
			&\ds =c_N\frac{N-4}{2} \varepsilon^2 \log{\frac{1}{\varepsilon}}\int_{B(0, \frac{\rho}{\sqrt{\varepsilon}}) } \frac{1}{(1+|y|^2)^{N-4}}dy +\varepsilon^2 \ds \int_{B(0, \frac{\rho}{\sqrt{\varepsilon}}) }  \frac{c_N}{(1+|y|^2)^{N-4}}
			\log {\frac{c_N}{(1+|y|^2)^{N-4}}}dy\\[4mm]
			&=\ds c_N\frac{N-4}{2} \varepsilon^2 \log{\frac{1}{\varepsilon}}\int_{B(0, \frac{\rho}{\sqrt{\varepsilon}}) } \frac{1}{(1+|y|^2)^{N-4}}dy +\varepsilon^2 c_N \log c_N\int_{B(0, \frac{\rho}{\sqrt{\varepsilon}}) } \frac{1}{(1+|y|^2)^{N-4}}dy\\[4mm]
			&+\ds \varepsilon^2 c_N\ds \int_{B(0, \frac{\rho}{\sqrt{\varepsilon}}) }  \frac{1}{(1+|y|^2)^{N-4}}  \log{\frac{1}{(1+|y|^2)^{N-4}}}dy\\[4mm]
			&=\ds c_N \cdot K_2\cdot \frac{N-4}{2} \varepsilon^2 \log{\frac{1}{\varepsilon}} +O(\varepsilon^2)+\ds \varepsilon^2 c_N O\left(\ds \int_{\mathbb{R} }  \frac{1}{(1+|y|^2)^{N-4-\frac{1}{3}}}dy\right)\\[3mm]
			&=\ds c_N \cdot K_2\cdot \frac{N-4}{2} \varepsilon^2 \log{\frac{1}{\varepsilon}} +O(\varepsilon^2),
			\\[3mm]
		\end{array}
		$$
		where we use the following fact that
		$$
		\ds \int_{B^c(0, \frac{\rho}{\sqrt{\varepsilon}}) } \frac{1}{(1+|y|^2)^{N-4}}dy=O(\varepsilon ^{\frac{N-8}{2}}).
		$$
		Hence 
		$$
		\ds \int_\Omega V_\varepsilon ^2 \log {V_\varepsilon^2}=c_N K_2\frac{N-4}{2} \varepsilon^2 \log{\frac{1}{\varepsilon}}+O(\varepsilon^2).
		$$
		%
	\end{proof}
	\begin{lemma}\label{lem3.4}
		Assume that $N>8$. If $\mu<0$, $\tau\in\mathbb{R}$ or $\mu=0$,  $\tau>0$, then we have
		\[
		\underset{t\ge 0}{\mathop{\sup }}\,I(tV_\varepsilon)<\frac{2}{N}S^\frac{N}{4}.
		\]
	\end{lemma}
	\begin{proof}
		By Lemma \ref{lem2.2}, $I(0)=0$ and $\underset{t\to +\infty }{\mathop{\lim }}\,~I(t V _\varepsilon )=-\infty$, we can choose $ t_\varepsilon \in (0,+\infty)$ such that
		\begin{equation*}
			I(t_\varepsilon V_\varepsilon )=\underset{t\ge 0}{\mathop{\sup }}\,I(tV_\varepsilon),
		\end{equation*}
		and
		$$
		\ds \int(\left|\Delta V_\varepsilon\right|^2+\mu |\nabla V_\varepsilon|^2-\lambda |V_\varepsilon|^2)-\ds t_\varepsilon^{2^{**}-2}\int |V_\varepsilon|^{2^{**}}-\tau \log t_\varepsilon^2\int V_\varepsilon^2-\tau \ds \int V_\varepsilon^2\log{V_\varepsilon^2}=0,
		$$
		which implies that, as $\varepsilon \to 0^+$
		\begin{equation*}
			\begin{array}{ll}
				t_\varepsilon^{2^{**}-2}&=\frac{\ds \int(\left|\Delta V_\varepsilon\right|^2+\mu |\nabla V_\varepsilon|^2-\lambda |V_\varepsilon|^2)-\tau \log t_\varepsilon^2\int V_\varepsilon^2-\tau \ds \int V_\varepsilon^2\log{V_\varepsilon^2}}{\ds \int |V_\varepsilon|^{2^{**}}}\\[3mm]
				&=\ds \frac{S^\frac{N}{4}+O(\varepsilon)}
				{S^\frac{N}{4}+O(\varepsilon^\frac{N}{2})}\to 1.\\[3mm]	
			\end{array}
		\end{equation*}
		Thus, together with Lemma \ref{lem3.1}, we can see that
		$$
		\tau \log t_{\varepsilon}^2\int|V_\varepsilon|^2=o(\varepsilon^2).
		$$
		
		When $N>8$, by Lemma \ref{lem3.1} and Lemma \ref{lem3.2}, we have that, if $\mu<0$, $\tau\in\mathbb{R}$ or $\mu=0$, $\tau>0$, then, as $\varepsilon \to 0^+$,
		\begin{equation*}
			\begin{array}{ll}
				I(t_\varepsilon V_\varepsilon )
				&=\ds \frac{t_\varepsilon^2}{2} \ds(|\Delta V_\varepsilon|^2_2 +\mu |\nabla V_\varepsilon|^2_2-\lambda |V_\varepsilon|^2_2) -\frac{t_\varepsilon^{ 2^{**} } } {2^{**}} |V_\varepsilon|^ {2^{**}}_{2^{**}}-\ds \frac{\tau}{2}t^2_\varepsilon(\log t^2_\varepsilon -1)|V_\varepsilon|^2_2-\frac{\tau}{2}t_\varepsilon^2 \ds \int_\Omega V_\varepsilon^2\log V^2_\varepsilon  \\[3mm]
				& =  \ds  (\frac{t_\varepsilon^2}{2}-\frac{t_\varepsilon^{ 2^{**} } } {2^{**}} ) S^\frac{N}{4} +O(\varepsilon ^\frac{N-4}{2})+\mu\frac{t_{\varepsilon }^{2}}{2}C_N K_1\varepsilon-\lambda\frac{t_{\varepsilon }^{2}}{2} c_N K_2 \varepsilon^2\\[3mm] 
				&+\ds o(\varepsilon^2)+\frac{\tau}{2}t_\varepsilon^2c_N K_2 \varepsilon^2 -\frac{\tau}{2}t^2_\varepsilon \left(c_N \frac{N-4}{2}K_2\varepsilon^2\log\frac{1}{\varepsilon}+O(\varepsilon^2)\right) \\[3mm]
				& \leq \ds \frac{2}{N} S^\frac{N}{4}+\mu\frac{t_{\varepsilon }^{2}}{2}C_N K_1\varepsilon-\tau\frac{t^2_\varepsilon}{2} c_N \frac{N-4}{2}K_2\varepsilon^2\log\frac{1}{\varepsilon}+O(\varepsilon^2)\\[3mm]
				&<\ds \frac{2}{N} S^\frac{N}{4}.
			\end{array}
		\end{equation*}
		%
	\end{proof}
	
	\textbf{The case of} $\mathbf{N=8:}$
	
	\begin{lemma}\label{lem3.5}	
		For $N=8$, we have that, as $\varepsilon\to0^+$
		$$
		\int_\Omega V_\varepsilon^2 \log V_\varepsilon ^2 \ge c_8\log\left(\frac{\sqrt{c_8}(\varepsilon+\rho^2)}{e^{\frac{11}{6}}(\varepsilon+4\rho^2)^2}\right) \omega_8\varepsilon^2 \log \frac{1}{\varepsilon} +O(\varepsilon^2),
		$$
		and
		$$
		\int_\Omega V_\varepsilon^2 \log V_\varepsilon ^2 \leq c_8\log\left(\frac{\sqrt{c_8}e^{\frac{25}{6}}(\varepsilon+4\rho^2)}{(\varepsilon+\rho^2)^2}\right) \omega_8\varepsilon^2 \log \frac{1}{\varepsilon} +O(\varepsilon^2),
		$$
		where $c_8$ and $\omega_8$ have been given in Lemma \ref{lem3.1}.
		
	\end{lemma}
	\begin{proof}
		For $N=8$, we deduce that
		\begin{equation}\label{eqS3.3}
			\begin{array}{ll}
				\ds \int_\Omega V_\varepsilon^2 \log V_\varepsilon ^2 &=\ds \int_\Omega \varphi^2u_\varepsilon^2\log(\varphi^2 u_{\varepsilon}^2)dx\\[3mm]
				&=c_8 \ds \int_\Omega \varphi ^2 \frac{\varepsilon^2}{(\varepsilon +|x|^2)^4} \log\left[c_8 \varphi ^2 \frac{\varepsilon^2}{(\varepsilon +|x|^2)^4}\right]dx\\[3mm]
				&=c_8 \ds \int_\Omega \varphi ^2 \frac{\varepsilon^2}{(\varepsilon +|x|^2)^4} \log \frac{\varepsilon^2}{(\varepsilon +|x|^2)^4}dx+c_8\log c_8\ds \int_\Omega \varphi ^2 \frac{\varepsilon^2}{(\varepsilon+|x|^2)^4}dx\\[3mm]
				&+c_8 \ds \int_{\Omega \setminus\ B_\rho (0)} \varphi ^2 \log \varphi ^2 \frac{\varepsilon^2}{(\varepsilon+|x|^2)^4}dx\\[3mm]
				&=\ds I+II+O(\varepsilon^2).
			\end{array}
		\end{equation}
		A straightforward computations give us
\begin{align}
				II\notag&= c_8 \log c_8 \ds \int_{B_\rho(0) }\frac{\varepsilon^2}{(\varepsilon +|x|^2)^4}dx+O(\varepsilon^2)\\ \notag
				&=c_8 (\log c_8) \varepsilon^2 \ds \int_{B(0,\frac{\rho}{\sqrt{\varepsilon}}) }\frac{1}{(1+|y|^2)^4}dy+O(\varepsilon^2)\\ \label{eqS3.4}
				&=c_8 (\log c_8) \omega_8\varepsilon^2 \ds \int^{\frac{\rho}{\sqrt{\varepsilon}}}_0 \frac{1}{(1+r^2)^4}r^7dr+O(\varepsilon^2)\\ \notag
				&=\ds \frac{1}{2}c_8 (\log c_8) \omega_8 \varepsilon^2 \int^{\frac{\rho}{\sqrt{\varepsilon}}}_0\left[ \frac{1}{1+r^2}-\frac{3}{(1+r^2)^2}+\frac{3}{(1+r^2)^3}-\frac{1}{(1+r^2)^4}\right] d(1+r^2)+O(\varepsilon^2)\\ \notag
				&=\ds \frac{1}{2}c_8 (\log c_8) \omega_8 \varepsilon^2 \log \frac{1}{\varepsilon}+O(\varepsilon^2),
\end{align}
		and
		\begin{equation}\label{eqS3.5}
			\begin{array}{ll}
				I&=c_8 \ds \int_\Omega \varphi^2 \frac{\varepsilon^2}{(\varepsilon +|x|^2)^4}\log \frac{\varepsilon^2}{(\varepsilon +|x|^2)^4}dx\\[3mm]
				&=c_8  \ds \int_{B_{2\rho}(0) }\varphi^2 \frac{\varepsilon^2}{(\varepsilon +|x|^2)^4}\log \frac{\varepsilon^2}{(\varepsilon +|x|^2)^4}dx\\[3mm]
				&\ds=c_8 \varepsilon^2 \ds \int_{B(0,\frac{2\rho}{\sqrt{\varepsilon}})} \varphi^2(\sqrt{\varepsilon }y)\frac{1}{(1+|y|^2)^4}\log\frac{1}{\varepsilon^2(1+|y|^2)^4}dy\\[3mm]
				&\ds =2c_8\varepsilon^2\log \frac{1}{\varepsilon}\ds \int_{B(0,\frac{2 \rho}{\sqrt{\varepsilon}}) }\varphi^2(\sqrt{\varepsilon }y)\frac{1}{(1+|y|^2)^4}dy\\[3mm]
				&\ds +c_8 \varepsilon^2 \ds \int_{B(0,\frac{2\rho}{\sqrt{\varepsilon}})} \varphi^2(\sqrt{\varepsilon }y)\frac{1}{(1+|y|^2)^4}\log \frac{1}{(1+|y|^2)^4}dy\\[3mm]
				&\ds =I_1+I_2,
			\end{array}
		\end{equation}
		where
		\begin{equation}\label{eqS3.6}
			\begin{array}{ll}
				I_1&\ge\ds 2c_8 \omega_8\varepsilon^2\log\frac{1}{\varepsilon} \int^{\frac{\rho}{\sqrt{\varepsilon}}}_0 \frac{1}{(1+r^2)^4}r^7dr\\[3mm]
				&=\ds c_8\omega_8 \varepsilon^2\log\frac{1}{\varepsilon} \left[\log(\varepsilon+\rho^2)+\log\frac{1}{\varepsilon}-\frac{11}{6}+\frac{3\varepsilon}{\varepsilon+\rho^2}-\frac{3\varepsilon^2}{2(\varepsilon+\rho^2)^2}+\frac{\varepsilon^3}{3(\varepsilon+\rho^2)^3}\right]\\[3mm]
				&=\ds c_8\omega_8 \varepsilon^2\left(\log\frac{1}{\varepsilon}\right)^2+c_8\omega_8\log\left(\frac{\varepsilon+\rho^2}{e^{\frac{11}{6}}}\right)\varepsilon^2\log\frac{1}{\varepsilon}+O(\varepsilon^3\log\frac{1}{\varepsilon}),
			\end{array}
		\end{equation}
		
		\begin{equation}\label{eqS3.7}
			\begin{array}{ll}
				I_1&\ds \le \ds 2c_8 \omega_8\varepsilon^2\log\frac{1}{\varepsilon} \int^{\frac{2\rho}{\sqrt{\varepsilon}}}_0 \frac{1}{(1+r^2)^4}r^7dr\\[3mm]
				&=\ds c_8\omega_8 \varepsilon^2\log\frac{1}{\varepsilon} \left[\log(\varepsilon+4\rho^2)+\log\frac{1}{\varepsilon}-\frac{11}{6}+\frac{3\varepsilon}{\varepsilon+4\rho^2}-\frac{3\varepsilon^2}{2(\varepsilon+4\rho^2)^2}+\frac{\varepsilon^3}{3(\varepsilon+4\rho^2)^3}\right]\\[3mm]
				&=\ds c_8\omega_8 \varepsilon^2\left(\log\frac{1}{\varepsilon}\right)^2+c_8\omega_8\log\left(\frac{\varepsilon+4\rho^2}{e^{\frac{11}{6}}}\right)\varepsilon^2\log\frac{1}{\varepsilon}+O(\varepsilon^3\log\frac{1}{\varepsilon}),
			\end{array}
		\end{equation}
	\begin{align}
				I_2\ \notag &\ds \ge \ds -4c_8 \varepsilon^2\int_{B(0,\frac{2 \rho}{\sqrt{\varepsilon}})}\frac{1}{(1+|y|^2)^4}\log {(1+|y|^2)}dy\\ \notag
				&=\ds -4 c_8\omega_8\varepsilon^2\int^{\frac{2\rho}{\sqrt{\varepsilon}}}_0\frac{r^7}{(1+r^2)^4}\log(1+r^2)dr\\  \notag
				&= \ds -2c_8\omega_8\varepsilon^2\int^{\frac{2\rho}{\sqrt{\varepsilon}}}_0\frac{(1+r^2)^3-3r^2(1+r^2)-1}{(1+r^2)^4}\log(1+r^2)d(1+r^2)\\  \notag
				&=\ds -2c_8\omega_8\varepsilon^2\int^{\frac{2\rho}{\sqrt{\varepsilon}}}_0\frac{1}{1+r^2}\log(1+r^2)d(1+r^2)\\  \label{eqS3.8}
				&+\ds 2c_8\omega_8\varepsilon^2\int^{\frac{2\rho}{\sqrt{\varepsilon}}}_0\frac{3r^2(1+r^2)+1}{(1+r^2)^4}\log(1+r^2)d(1+r^2)\\  \notag
				&\ge \ds -c_8\omega_8\varepsilon^2\left(\log(1+r^2)\right)^2{\Big\arrowvert}^{\frac{2\rho}{\sqrt{\varepsilon}}}_0\\  \notag
				&= \ds -c_8\omega_8\varepsilon^2\left(\log(1+\frac{4\rho^2}{\varepsilon})\right)^2\\  \notag
				&= \ds -c_8\omega_8\varepsilon^2\left[\log(\varepsilon+4\rho^2)+\log\frac{1}{\varepsilon}\right]^2\\  \notag
				&=\ds -c_8\omega_8\varepsilon^2\left(\log\frac{1}{\varepsilon}\right)^2-2c_8\omega_8\log(\varepsilon+4\rho^2)\varepsilon^2\log\frac{1}{\varepsilon}+O(\varepsilon^2),
			\end{align}
		and
		\begin{equation}\label{eqS3.9}
			\begin{array}{ll}
				I_2&\le \ds -4c_8 \varepsilon^2\int_{B(0,\frac{ \rho}{\sqrt{\varepsilon}})}\frac{1}{(1+|y|^2)^4}\log {(1+|y|^2)}dy\\[3mm]
				&=\ds -4 c_8\omega_8\varepsilon^2\int^{\frac{\rho}{\sqrt{\varepsilon}}}_0\frac{r^7}{(1+r^2)^4}\log(1+r^2)dr\\[3mm]
				&= \ds -2c_8\omega_8\varepsilon^2\int^{\frac{\rho}{\sqrt{\varepsilon}}}_0\frac{(1+r^2)^3-3r^2(1+r^2)-1}{(1+r^2)^4}\log(1+r^2)d(1+r^2)\\[3mm]
				&\le \ds -2c_8\omega_8\varepsilon^2\int^{\frac{\rho}{\sqrt{\varepsilon}}}_0\frac{1}{1+r^2}\log(1+r^2)d(1+r^2)+\ds 2c_8\omega_8\varepsilon^2\int^{\frac{\rho}{\sqrt{\varepsilon}}}_0\frac{3r^2}{(1+r^2)^2}+\frac{1}{(1+r^2)^3}d(1+r^2)\\[3mm]
				&= \ds -c_8\omega_8\varepsilon^2\left(\log(1+\frac{\rho^2}{\varepsilon})\right)^2+6c_8\omega_8\varepsilon^2\int^{\frac{\rho}{\sqrt{\varepsilon}}}_0\frac{1}{1+r^2}d(1+r^2)+O(\varepsilon^3)\\[3mm]
				&=\ds -c_8\omega_8\varepsilon^2\left(\log\frac{1}{\varepsilon}\right)^2-2c_8\omega_8\log(\varepsilon+\rho^2)\varepsilon^2\log\frac{1}{\varepsilon}+6c_8\omega_8\varepsilon^2\log(\frac{\varepsilon+\rho^2}{\varepsilon})+O(\varepsilon^2)\\[3mm]
				&=\ds  -c_8\omega_8\varepsilon^2\left(\log\frac{1}{\varepsilon}\right)^2-2c_8\omega_8\log\left(\frac{\varepsilon+\rho^2}{e^3}\right)\varepsilon^2\log\frac{1}{\varepsilon}+O(\varepsilon^2).
			\end{array}
		\end{equation}
		Therefore, according to $\eqref{eqS3.3}-\eqref{eqS3.9}$, one has
		\begin{equation*}
			\begin{array}{ll}
				\ds\int_\Omega V_\varepsilon^2 \log V_\varepsilon ^2 &\ge \ds \frac{1}{2}c_8 (\log c_8) \omega_8 \varepsilon^2 \log \frac{1}{\varepsilon}+c_8\omega_8 \varepsilon^2\left(\log\frac{1}{\varepsilon}\right)^2+c_8\omega_8\log\left(\frac{\varepsilon+\rho^2}{e^{\frac{11}{6}}}\right)\varepsilon^2\log\frac{1}{\varepsilon}\\[3mm]
				&-\ds c_8\omega_8\varepsilon^2\left(\log\frac{1}{\varepsilon}\right)^2-2c_8\omega_8\log(\varepsilon+4\rho^2)\varepsilon^2\log\frac{1}{\varepsilon}+O(\varepsilon^2)\\[3mm]
				&=\ds c_8\log\left(\frac{\sqrt{c_8}(\varepsilon+\rho^2)}{e^{\frac{11}{6}}(\varepsilon+4\rho^2)^2}\right) \omega_8\varepsilon^2 \log \frac{1}{\varepsilon} +O(\varepsilon^2),
			\end{array}
		\end{equation*}
		and
		\begin{equation*}
			\begin{array}{ll}
				\ds\int_\Omega V_\varepsilon^2 \log V_\varepsilon ^2 &\le \ds \frac{1}{2}c_8 (\log c_8) \omega_8 \varepsilon^2 \log \frac{1}{\varepsilon}+c_8\omega_8 \varepsilon^2\left(\log\frac{1}{\varepsilon}\right)^2+c_8\omega_8\log\left(\frac{\varepsilon+4\rho^2}{e^{\frac{11}{6}}}\right)\varepsilon^2\log\frac{1}{\varepsilon}\\[3mm]
				&-\ds c_8\omega_8\varepsilon^2\left(\log\frac{1}{\varepsilon}\right)^2-2c_8\omega_8\log\left(\frac{\varepsilon+\rho^2}{e^3}\right)\varepsilon^2\log\frac{1}{\varepsilon}+O(\varepsilon^2)\\[3mm]
				&=\ds c_8\log\left(\frac{\sqrt{c_8}e^{\frac{25}{6}}(\varepsilon+4\rho^2)}{(\varepsilon+\rho^2)^2}\right) \omega_8\varepsilon^2 \log \frac{1}{\varepsilon} +O(\varepsilon^2).
			\end{array}
		\end{equation*}
		%
	\end{proof}
	
	\begin{lemma}\label{lem3.6}
		Assume that $N=8$. $\underset{t\ge 0}{\mathop{\sup }}\,I(tV_\varepsilon)<\frac{2}{N}S^\frac{N}{4}$ provided one of the following assumptions holds:
		
		(i) $\mu<0$, $\tau\in\mathbb{R}$;
		
		(ii) $\mu=0$, $\tau>0$;
		
		(iii)$\mu=0$, $\tau<0$ and  $\frac{\sqrt{c_8}4e^{\frac{\lambda}{2\tau}+\frac{19}{6}}}{\rho_{max}^2}<1$.
		
	\end{lemma}
	\begin{proof}	
		For $N=8$, similar to the case of $N>8$, we can find a $t_\varepsilon \in(0,+\infty)$ such that
		\begin{equation*}
			I(t_\varepsilon V_\varepsilon )=\underset{t\ge 0}{\mathop{\sup }}\,I(tV_\varepsilon),
		\end{equation*}
		and
		$$
		\tau \log t_{\varepsilon}^2\int|V_\varepsilon|^2=o(\varepsilon^2\log\frac{1}{\varepsilon}).
		$$
		
		\textbf{Case} $\mathbf{1}$: \ $\mu<0$, $\tau\in\mathbb{R}$.
		
		By Lemma \ref{lem3.5}, we could obtain
		$$\left|\int_\Omega V_\varepsilon^2\log V_\varepsilon^2\right|\leq O(\varepsilon^2\log\frac{1}{\varepsilon}).$$
		If $\varepsilon>0$ is small enough, we have
		\begin{equation*}
			\begin{array}{ll}
				I(t_\varepsilon V_\varepsilon )
				&=\ds \frac{t_\varepsilon^2}{2} \ds(|\Delta V_\varepsilon|^2_2 +\mu |\nabla V_\varepsilon|^2_2-\lambda |V_\varepsilon|^2_2) -\frac{t_\varepsilon^{ 2^{**} } } {2^{**}} |V_\varepsilon|^ {2^{**}}_{2^{**}}-\ds \frac{\tau}{2}t^2_\varepsilon(\log t^2_\varepsilon -1)|V_\varepsilon|^2_2-\frac{\tau}{2}t_\varepsilon^2 \ds \int_\Omega V_\varepsilon^2\log V^2_\varepsilon \\[3mm]
				&\leq \ds \frac{t_\varepsilon^2}{2}S^{\frac{N}{4}}-\frac{t_\varepsilon^2}{2^{**}}S^{\frac{N}{4}}+\frac{t_{\varepsilon }^{2}}{2} \mu C_8 K_1\varepsilon+O(\varepsilon^2\log\frac{1}{\varepsilon})\\[3mm]
				& \leq \ds  \frac{2}{N} S^\frac{N}{4} +\frac{t_{\varepsilon }^{2}}{2} \mu C_8 K_1\varepsilon+O(\varepsilon^2\log\frac{1}{\varepsilon})\\[3mm]
				&<\ds\frac{2}{N} S^\frac{N}{4}.
			\end{array}
		\end{equation*}
		This implies that
		\[
		\underset{t\ge 0}{\mathop{\sup }}\,I(tV_\varepsilon)<\frac{2}{N}S^\frac{N}{4}.
		\]
		
		\textbf{Case} $\mathbf{2}$: \ $\mu=0$, $\tau>0$.
		
		If $\varepsilon>0$ is small enough, we obtain
		\begin{align*}
				I(t_\varepsilon V_\varepsilon )
				&=\ds \frac{t_\varepsilon^2}{2} \ds(|\Delta V_\varepsilon|^2_2 -\lambda |V_\varepsilon|^2_2) -\frac{t_\varepsilon^{ 2^{**} } } {2^{**}} |V_\varepsilon|^ {2^{**}}_{2^{**}}-\ds \frac{\tau}{2}t^2_\varepsilon(\log t^2_\varepsilon -1)|V_\varepsilon|^2_2-\frac{\tau}{2}t_\varepsilon^2 \ds \int_\Omega V_\varepsilon^2\log V^2_\varepsilon \\
				& \leq \ds  \frac{2}{N} S^\frac{N}{4} -\frac{\lambda}{4} t_\varepsilon^2 c_8\omega_8 \varepsilon^2 \log \frac{1}{\varepsilon}-\frac{\tau}{4}t_\varepsilon^2\log t_\varepsilon^2 c_8\omega_8 \varepsilon^2 \log \frac{1}{\varepsilon}+ O(\varepsilon^2) \\
				&+\ds \frac{\tau}{4} t_\varepsilon^2 c_8\omega_8\varepsilon^2 \log \frac{1}{\varepsilon}-\frac{\tau}{2}t_\varepsilon^2\left(c_8\log\left(\frac{\sqrt{c_8}(\varepsilon+\rho^2)}{e^{\frac{11}{6}}(\varepsilon+4\rho^2)^2}\right) \omega_8\varepsilon^2 \log \frac{1}{\varepsilon}\right) \\
				& \leq \ds \frac{2}{N} S^\frac{N}{4} -\frac{t_\varepsilon^2}{2}c_8\omega_8\varepsilon^2 \log \frac{1}{\varepsilon}\left[\frac{\lambda-\tau}{2}+\tau\log(\frac{\sqrt{c_8}(\varepsilon+\rho^2)}{e^{\frac{11}{6}}(\varepsilon+4\rho^2)^2})\right]+o(\varepsilon^2 \log \frac{1}{\varepsilon})\\
				&=\ds \frac{2}{N} S^\frac{N}{4} -\frac{t_\varepsilon^2}{2}c_8\omega_8\varepsilon^2 \log \frac{1}{\varepsilon}\left(\log(\frac{{c_8}^{\frac{\tau}{2}}(\varepsilon+\rho^2)^{\tau}}{e^{\frac{14\tau-3\lambda}{6}}(\varepsilon+4\rho^2)^{2\tau}})\right)+o(\varepsilon^2 \log \frac{1}{\varepsilon})\\
				&\leq \ds\frac{2}{N} S^\frac{N}{4} -\frac{t_\varepsilon^2}{2}c_8\omega_8\varepsilon^2 \log \frac{1}{\varepsilon}\log\left(\frac{{c_8}^{\frac{\tau}{2}}}{25^\tau e^{\frac{14\tau-3\lambda}{6}}\rho^{2\tau}}\right)+o(\varepsilon^2 \log \frac{1}{\varepsilon})\\
				&< \ds\frac{2}{N} S^\frac{N}{4},
		\end{align*}
		where we choose $\rho>0$ small enough such that $\frac{{c_8}^{\frac{\tau}{2}}}{25^\tau e^{\frac{14\tau-3\lambda}{6}}\rho^{2\tau}}>1$.
		This implies that
		\[
		\underset{t\ge 0}{\mathop{\sup }}\,I(tV_\varepsilon)<\frac{2}{N}S^\frac{N}{4}.
		\]
		
		\textbf{Case} $\mathbf{3}$: \ $\mu=0$, $\tau<0$ and $\frac{\sqrt{c_8}4e^{\frac{\lambda}{2\tau}+\frac{19}{6}}}{\rho_{max}^2}<1$.
		
		We choose $\rho=\rho_{max}$. Similar to Case 2, for small enough $\varepsilon>0$, we have
		\begin{equation*}
			\begin{array}{ll}
				I(t_\varepsilon V_\varepsilon )
				&=\ds \frac{t_\varepsilon^2}{2} \ds(|\Delta V_\varepsilon|^2_2 -\lambda |V_\varepsilon|^2_2) -\frac{t_\varepsilon^{ 2^{**} } } {2^{**}} |V_\varepsilon|^ {2^{**}}_{2^{**}}-\ds \frac{\tau}{2}t^2_\varepsilon(\log t^2_\varepsilon -1)|V_\varepsilon|^2_2-\frac{\tau}{2}t_\varepsilon^2 \ds \int_\Omega V_\varepsilon^2\log V^2_\varepsilon \\[3mm]
				
				& \leq \ds  \frac{2}{N} S^\frac{N}{4}-\left(\frac{\lambda}{4}-\frac{\tau}{4}\right) t_\varepsilon^2 c_8\omega_8 \varepsilon \log \frac{1}{\varepsilon} -\ds \frac{\tau}{2}t_\varepsilon^2 c_8\log\left(\frac{\sqrt{c_8}e^{\frac{25}{6}}(\varepsilon+4\rho^2)}{(\varepsilon+\rho^2)^2}\right)\omega_8\varepsilon^2 \log \frac{1}{\varepsilon}+o(\varepsilon^2 \log \frac{1}{\varepsilon}) \\[3mm]
				& = \ds \frac{2}{N} S^\frac{N}{4} -\frac{t_\varepsilon^2}{2}\tau c_8\omega_8\varepsilon^2 \log \frac{1}{\varepsilon}\left[(\frac{\lambda}{2\tau}-1)+\log(\frac{\sqrt{c_8}e^{\frac{25}{6}}(\varepsilon+4\rho^2)}{(\varepsilon+\rho^2)^2})\right]+\ds o(\varepsilon^2 \log \frac{1}{\varepsilon})\\[3mm]
				&\leq \ds\frac{2}{N} S^\frac{N}{4} -\frac{t_\varepsilon^2}{2}\tau c_8\omega_8\varepsilon^2 \log \frac{1}{\varepsilon}\log\left(\frac{\sqrt{c_8}4e^{\frac{\lambda}{2\tau}+\frac{19}{6}}}{\rho^2}\right)+o(\varepsilon^2 \log \frac{1}{\varepsilon})\\[3mm]
				&<\ds\frac{2}{N} S^\frac{N}{4}.
			\end{array}
		\end{equation*}
		This implies that
		\[
		\underset{t\ge 0}{\mathop{\sup }}\,I(tV_\varepsilon)<\frac{2}{N}S^\frac{N}{4}.
		\]
		%
	\end{proof}
	
	\textbf{The case of} $\mathbf{5\leq N\leq 7:}$
	
	
	\begin{lemma}\label{lem3.7}
		If $5\leq N\leq 7$, then we have that, as $\varepsilon\to 0^+$
		$$
		\ds \int_\Omega V_\varepsilon^2 \log V_\varepsilon ^2=\ds -\frac{N-4}{2} c_N \omega_N \varepsilon^{\frac{N-4}{2}} \log \frac{1}{\varepsilon}\int_0^{2\rho}\varphi^2r^{7-N}dr+O(\varepsilon^{\frac{N-4}{2}}).
		$$
		where $c_N$ and $\omega_N$ have been given in Lemma \ref{lem3.1}.
	\end{lemma}
	\begin{proof}
		Following the definition of $V_\varepsilon$, one has
			\begin{align}
				\ds \int_\Omega V_\varepsilon^2 \log V_\varepsilon ^2 \notag &=\ds \int_\Omega \varphi^2u_\varepsilon^2\log(\varphi^2 u_{\varepsilon}^2)dx\\ \notag
				&=c_N \ds \int_{B(0,2\rho)} \varphi ^2 \frac{\varepsilon^{\frac{N-4}{2}}}{(\varepsilon +|x|^2)^{N-4}} \log\left[c_N \varphi ^2 \frac{\varepsilon^{\frac{N-4}{2}}}{(\varepsilon +|x|^2)^{N-4}}\right]dx\\ \notag
				&=c_N \ds \int_{B(0,2\rho)} \varphi ^2 \frac{\varepsilon^{\frac{N-4}{2}}}{(\varepsilon +|x|^2)^{N-4}} \left(\log c_N+\log \varphi^2+\log \frac{\varepsilon^{\frac{N-4}{2}}}{(\varepsilon +|x|^2)^{N-4}}\right)dx\\ \label{eqS3.10}
				&=c_N \ds \int_{B(0,2\rho)} \varphi ^2 \frac{\varepsilon^{\frac{N-4}{2}}}{(\varepsilon +|x|^2)^{N-4}}\log \frac{\varepsilon^{\frac{N-4}{2}}}{(\varepsilon +|x|^2)^{N-4}}dx+O(\varepsilon^{\frac{N-4}{2}})\\ \notag
				&=\ds-\frac{N-4}{2} c_N \varepsilon^{\frac{N-4}{2}} \log \frac{1}{\varepsilon}\int_{B(0,2\rho)}\varphi^2\frac{1}{(\varepsilon +|x|^2)^{N-4}}dx\\ \notag
				&+\ds c_N \varepsilon^{\frac{N-4}{2}}\int_{B(0,2\rho)}\varphi^2\frac{1}{(\varepsilon +|x|^2)^{N-4}}\log\frac{1}{(\varepsilon +|x|^2)^{N-4}}dx +O(\varepsilon^{\frac{N-4}{2}})\\ \notag
				&=\ds I+II+O(\varepsilon^{\frac{N-4}{2}}).
			\end{align}
		Next, we will verify the estimation for $N=7$ as an example.
		%
		%
		By direct computation, we have 
				\begin{align}
				I\notag&=\ds -\frac{3}{2}c_7\varepsilon^{\frac{3}{2}}\log\frac{1}{\varepsilon}\int_{B(0,2\rho)}\varphi^2\frac{1}{(\varepsilon +|x|^2)^3}dx\\ \notag
				&=\ds -\frac{3}{2}c_7\omega_7\varepsilon^{\frac{3}{2}}\log\frac{1}{\varepsilon}\int_{0}^{2\rho}\varphi^2\frac{r^6}{(\varepsilon+r^2)^3}dr\\ \label{eqS3.11}
				&=\ds -\frac{3}{2}c_7\omega_7\varepsilon^{\frac{3}{2}}\log\frac{1}{\varepsilon}\int_{0}^{2\rho}\varphi^2\left(1-\frac{\varepsilon^3+3\varepsilon^2r^2+3\varepsilon r^4}{(\varepsilon+r^2)^3}\right)dr\\ \notag
				&=\ds -\frac{3}{2}c_7\omega_7\varepsilon^{\frac{3}{2}}\log\frac{1}{\varepsilon}\int_{0}^{2\rho}\varphi^2 dr -\frac{3}{2}c_7\omega_7\varepsilon^{\frac{3}{2}}\log\frac{1}{\varepsilon}\int_{0}^{2\rho}\varphi^2 \frac{\varepsilon^3+3\varepsilon^2r^2+3\varepsilon r^4}{(\varepsilon+r^2)^3}dr\\ \notag
				&=\ds-\frac{3}{2}c_7\omega_7\varepsilon^{\frac{3}{2}}\log\frac{1}{\varepsilon}\int_{0}^{2\rho}\varphi^2 dr+O(\varepsilon^2\log\frac{1}{\varepsilon}),
				\end{align}
		and
		\begin{equation}\label{eqS3.12}
			\begin{array}{ll}
				|II|&=\ds c_7\varepsilon^{\frac{3}{2}}\int_{B(0,2\rho)}\varphi^2\frac{1}{(\varepsilon +|x|^2)^3}\log \frac{1}{(\varepsilon +|x|^2)^3}dx\\[3mm]
				&\le \ds c_7\omega_7\varepsilon^{\frac{3}{2}}\int_{0}^{\rho}\log \frac{1}{(\varepsilon+r^2)^3}dr+O(\varepsilon^{\frac{3}{2}})\\[3mm]
				&=\ds -3c_7\omega_7\varepsilon^{\frac{3}{2}}\int_{0}^{\rho}\log(\varepsilon+r^2)dr+O(\varepsilon^{\frac{3}{2}})\\[3mm]
				&=\ds -3c_7\omega_7\varepsilon^{\frac{3}{2}} r\log(\varepsilon+r^2){\Big\arrowvert}^{\rho}_0+3c_7\omega_7\varepsilon^{\frac{3}{2}}\int_{0}^{\rho}r\cdot\frac{1}{\varepsilon+r^2}\cdot2rdr+O(\varepsilon^{\frac{3}{2}})\\[3mm]
				&=\ds O(\varepsilon^{\frac{3}{2}}).
			\end{array}
		\end{equation}
		Hence, when $N=7$, combining with (\ref{eqS3.10})-(\ref{eqS3.12}), we get
		$$
		\ds \int_\Omega V_\varepsilon^2 \log V_\varepsilon ^2=\ds -\frac{3}{2} c_7 \omega_7 \varepsilon^{\frac{3}{2}} \log \frac{1}{\varepsilon}\int_0^{2\rho}\varphi^2dr+O(\varepsilon^{\frac{3}{2}}).
		$$

		For $N=5$ or $N=6$, we can use a similar  method as above to obtain the following estimation:
		\begin{equation*}
			\begin{array}{ll}
				\ds \int_\Omega V_\varepsilon^2 \log V_\varepsilon ^2=\left\{ \begin{matrix}
					\ds -c_6 \omega_6 \varepsilon \log \frac{1}{\varepsilon}\int_0^{2\rho}\varphi^2 rdr+O(\varepsilon),&for~N=6 , \\[3mm]
					\ds -\frac{1}{2}c_5 \omega_5 \varepsilon^{\frac{1}{2}} \log \frac{1}{\varepsilon}\int_0^{2\rho}\varphi^2 r^2dr+O(\varepsilon^{\frac{1}{2}}),&for~N=5. \\[3mm]
				\end{matrix} \right.
			\end{array}
		\end{equation*}
		We omit it here.
		%
	\end{proof}
	
	\begin{lemma}\label{lem3.8}
		If $5\le N \le7$, then $\underset{t\ge 0}{\mathop{\sup }}\,I(tV_\varepsilon)<\frac{2}{N}S^\frac{N}{4}$ provided one of the following assumptions holds:
		
		(i)$N=7$, $\mu<0$, $\tau\in\mathbb{R}$ or $\mu=0$, $\tau<0$;
		
		(ii)$N=6$, $\mu<0$, $\tau\in\mathbb{R}$ or $\mu=0$, $\tau<0$;
		
		(iii)$N=5$, $\tau<0$.
	\end{lemma}
	\begin{proof}
		
		Similar to the case of $N>8$, we can find a $t_\varepsilon \in(0,+\infty)$ such that
		\begin{equation*}
			I(t_\varepsilon V_\varepsilon )=\underset{t\ge 0}{\mathop{\sup }}\,I(tV_\varepsilon),
		\end{equation*}
		and
		$$
		\tau \log t_{\varepsilon}^2\int|V_\varepsilon|^2=o(\varepsilon^{\frac{N-4}{2}}).
		$$
		%
		%
		
		For $N=7$, by Lemma \ref{lem3.2}, we have, as $\varepsilon\to0^+$
		$$
		\ds\int_ \Omega  |\nabla V_\varepsilon |^2 dx\ge \ds \frac{9}{32}c_7\omega_7\varepsilon+O(\varepsilon^{\frac{3}{2}}).
		$$
		
		Hence, combining with Lemma \ref{lem3.1}, we get that, if $\mu<0$, $\tau\in\mathbb{R}$ or $\mu=0$, $\tau<0$, then as $\varepsilon\to0^+$
		\begin{equation*}
			\begin{array}{ll}
				I(t_\varepsilon V_\varepsilon )
				&=\ds \frac{t_\varepsilon^2}{2} \ds(|\Delta V_\varepsilon|^2_2 +\mu |\nabla V_\varepsilon|^2_2-\lambda |V_\varepsilon|^2_2) -\frac{t_\varepsilon^{ 2^{**} } } {2^{**}} |V_\varepsilon|^ {2^{**}}_{2^{**}}-\ds \frac{\tau}{2}t^2_\varepsilon(\log t^2_\varepsilon -1)|V_\varepsilon|^2_2-\frac{\tau}{2}t_\varepsilon^2 \ds \int_\Omega V_\varepsilon^2\log V^2_\varepsilon \\[3mm]
				&\leq \ds \frac{t_\varepsilon^2}{2}S^{\frac{N}{4}}+\frac{9t_\varepsilon^2}{64}\mu c_7\omega_7\varepsilon-\frac{t_\varepsilon^2}{2^{**}}S^{\frac{N}{4}}+\frac{t_\varepsilon^2}{2}\tau\cdot \frac{3}{2} c_7 \omega_7 \varepsilon^{\frac{3}{2}} \log \frac{1}{\varepsilon}\int_0^{2\rho}\varphi^2dr+O(\varepsilon^{\frac{3}{2}})\\[3mm]
				&\leq \ds \frac{2}{N} S^\frac{N}{4} +\frac{9t_\varepsilon^2}{64}\mu c_7\omega_7\varepsilon+\frac{t_\varepsilon^2}{2}\tau\cdot \frac{3}{2} c_7 \omega_7 \varepsilon^{\frac{3}{2}} \log \frac{1}{\varepsilon}\int_0^{2\rho}\varphi^2dr+O(\varepsilon^{\frac{3}{2}})\\[3mm]
				&< \ds \frac{2}{N} S^\frac{N}{4}.
			\end{array}
		\end{equation*}
		%
		%
		
		For $N=6$, combining with Lemma \ref{lem3.1} and Lemma \ref{lem3.2}, we have that, if $\mu<0$, $\tau\in\mathbb{R}$ or $\mu= 0$, $\tau<0$, then as $\varepsilon\to0^+$
		\begin{equation}\label{eqS3.13}
			\begin{array}{ll}
				I(t_\varepsilon V_\varepsilon )
				&=\ds \frac{t_\varepsilon^2}{2} \ds(|\Delta V_\varepsilon|^2_2 +\mu |\nabla V_\varepsilon|^2_2-\lambda |V_\varepsilon|^2_2) -\frac{t_\varepsilon^{ 2^{**} } } {2^{**}} |V_\varepsilon|^ {2^{**}}_{2^{**}}-\ds \frac{\tau}{2}t^2_\varepsilon(\log t^2_\varepsilon -1)|V_\varepsilon|^2_2-\frac{\tau}{2}t_\varepsilon^2 \ds \int_\Omega V_\varepsilon^2\log V^2_\varepsilon \\[3mm]
				&\leq \ds \frac{t_\varepsilon^2}{2}S^{\frac{N}{4}}+t_\varepsilon^2\mu c_6\omega_6\varepsilon\log \frac{1}{\varepsilon}-\frac{t_\varepsilon^2}{2^{**}}S^{\frac{N}{4}}+\frac{t_\varepsilon^2}{2}\tau c_6 \omega_6 \varepsilon\log \frac{1}{\varepsilon}\int_0^{2\rho}\varphi^2rdr+O(\varepsilon)\\[3mm]
				&\leq \ds \frac{2}{N} S^\frac{N}{4} +\frac{t_{\varepsilon }^{2}}{2}c_6 \omega_6 \varepsilon \log \frac{1}{\varepsilon}\left(2\mu+\tau \int_0^{2\rho}\varphi^2rdr\right)+O(\varepsilon)\\[3mm]
				&< \ds \frac{2}{N} S^\frac{N}{4},
			\end{array}
		\end{equation}
		where we pick an appropriate $\rho$ such that $2\mu+\tau\int_{0}^{2\rho}\varphi^2rdr<0$.
		%
		%
		
		For $N=5$, we obtain that, if $\tau<0$, then as $\varepsilon\to0^+$
	\begin{align*}
				I(t_\varepsilon V_\varepsilon )
				&=\ds \frac{t_\varepsilon^2}{2} \ds(|\Delta V_\varepsilon|^2_2 +\mu |\nabla V_\varepsilon|^2_2-\lambda |V_\varepsilon|^2_2) -\frac{t_\varepsilon^{ 2^{**} } } {2^{**}} |V_\varepsilon|^ {2^{**}}_{2^{**}}-\ds \frac{\tau}{2}t^2_\varepsilon(\log t^2_\varepsilon -1)|V_\varepsilon|^2_2-\frac{\tau}{2}t_\varepsilon^2 \ds \int_\Omega V_\varepsilon^2\log V^2_\varepsilon \\
				&\leq \ds \frac{t_\varepsilon^2}{2}S^{\frac{N}{4}}-\frac{t_\varepsilon^2}{2^{**}}S^{\frac{N}{4}}+\frac{t_\varepsilon^2}{4}\tau c_5 \omega_5 \varepsilon^{\frac{1}{2}}\log \frac{1}{\varepsilon}\int_0^{2\rho}\varphi^2r^2dr+O(\varepsilon^\frac{1}{2})\\
				&\leq \ds \frac{2}{N} S^\frac{N}{4} +\frac{t_\varepsilon^2}{4}\tau c_5 \omega_5 \varepsilon^{\frac{1}{2}}\log \frac{1}{\varepsilon}\int_0^{2\rho}\varphi^2r^2dr+O(\varepsilon^\frac{1}{2})\\
				&< \ds \frac{2}{N} S^\frac{N}{4}.
	\end{align*}
		%
	\end{proof}
	\section{The proof of the main results}
	
	In this section, we give the proof of Theorem \ref{t1.1}.
	\begin{proof}[\bf The Proof of Theorem \ref{t1.1}:]  \
		Assume that $N\ge5$, $\lambda$, $\mu$ and $\tau$ satisfy the assumptions of  Theorem \ref{t1.1}. By Lemma \ref{lem2.2}, we have that the functional $I$ has the
		Mountain Pass geometry  which implies that there exists a sequence $\{u_n\}\subset H^2_{0}(\Omega)$ such that, as $n \to \infty $
		$$I(u_n)\to d~and~I'(u_n)\to0.$$
		Using Lemma \ref{lem2.4}, we know that $\{u_n\}$ is bounded in $H^2_{0}(\Omega)$. It follows from Lemmas \ref{lem2.6}, \ref{lem2.7}, \ref{lem3.4}, \ref{lem3.6} and \ref{lem3.8}, then the equation $\eqref{eqS1.1}$ has at least a nontrivial weak solution.

		
		
	\end{proof}

	\textbf{Acknowledgments:} ~~
	This paper  was  supported by the fund from NSF of China (No. 12061012).

\end{document}